\documentclass[preprint,12pt]{article}
\usepackage[top=1.1in, bottom=1.2in, left=0.9in, right=0.9in]{geometry}
\usepackage{amsmath,amsfonts,amssymb,amsthm}
\usepackage{mathrsfs,dsfont}
\usepackage{graphicx,color}
\usepackage{colortbl,dcolumn}
\usepackage{psfrag}
\usepackage{booktabs}
\usepackage{cite}
\usepackage{subfigure}
\usepackage{url}
\allowdisplaybreaks[4]
\numberwithin{equation}{section}
\usepackage[justification=centering]{caption}

\renewcommand{\footnote}{}
\numberwithin{equation}{section}

\newcommand{\R}{\mathbb{R}}
\newcommand{\N}{\mathbb{N}}
\newcommand{\Z}{\mathbb{Z}}
\newcommand{\E}{\mathbb{E}}
\renewcommand{\P}{\mathbb{P}}
\newcommand{\tr}{\operatorname{trace}}

\newcommand{\diff}{{\,\rm{d}}}
\newcommand{\F}{\mathcal{F}}
\newtheorem{theorem}{Theorem}[section]

\newtheorem{lemma}[theorem]{Lemma}

\newtheorem{assumption}[theorem]{Assumption}

\begin{document}
\title{Stochastic theta methods for random periodic solution of stochastic differential equations under non-globally Lipschitz conditions\footnotemark[1]}

\author{Ziheng Chen\footnotemark[2],
Liangmin Cao\footnotemark[3], Lin Chen\footnotemark[4]}

\date{}

\maketitle
\footnotetext{\footnotemark[1] This work was supported by National Natural Science Foundation of China (Nos. 12201552, 11961029 and 12031020), the Natural Science Foundation of Jiangxi Province (No. 2020BABL201007), Yunnan Fundamental Research Projects (No. 202301AU070010) and Innovation Team of School of Mathematics and Statistics of Yunnan University (No. ST20210104).
}

\footnotetext{\footnotemark[2] School of Mathematics and Statistics, Yunnan University, Kunming, Yunnan 650500, China. Email: czh@ynu.edu.cn}

\footnotetext{\footnotemark[3] School of Mathematics and Statistics, Yunnan University, Kunming, Yunnan 650500, China. Email: caoliangmin@itc.ynu.edu.cn}

\footnotetext{\footnotemark[4] School of Statistics and Data Science, Jiangxi University of Finance and Economics, Nanchang, Jiangxi 330013, China. Email: chenlin@jxufe.edu.cn. Corresponding author.}

\begin{abstract}
      {\rm
           This work focuses on the numerical approximations of random periodic solutions of stochastic differential equations (SDEs). Under non-globally Lipschitz conditions, we prove the existence and uniqueness of random periodic solutions for the considered equations and its numerical approximations generated by the stochastic theta (ST) methods with $\theta \in (1/2,1]$. It is shown that the random periodic solution of each ST method converges strongly in the mean square sense to that of SDEs for all stepsize. More precisely, the mean square convergence order is $1/2$ for SDEs with multiplicative noise and $1$ for SDEs with additive noise. Numerical results are finally reported to confirm these theoretical findings.
      } \\


\end{abstract}

\section{Introduction}\label{introduction}

As a very active ongoing research area, the study of numerical solutions of SDEs has achieved a large amount of mathematical results in the past decades; see, e.g., monographs \cite{kloeden1992numerical, milstein2021stochastic} for SDEs with the traditional global Lipschitz condition and \cite{higham2002strong,
	hutzenthaler2011strong, hutzenthaler2012strong, wang2013tamed, zhang2013fundamental, hutzenthaler2015numerical, mao2015truncated, mao2016convergence, beyn2016stochastic, andersson2017meansquare, hutzenthaler2020perturbation,
	feng2017numerical, rong2020numerical,
	wu2023backward, guo2023order} for SDEs beyond such a restrictive condition.
However, just a very limited number of literature \cite{feng2017numerical, rong2020numerical, wu2023backward, guo2023order} began to study the numerical approximation of random periodic solutions of SDEs since the concept of random periodic solutions of random dynamical systems has been properly defined in \cite{zhao2009random}.
The periodicity of random periodic solutions of SDEs means that there exists a random invariant curve with period $\tau > 0$, which will be mapped into itself with respect to the shifted noise.
In general, the random invariant curves of random periodic solutions are rarely available in explicit form. Therefore it is natural to take advantage of numerical methods to characterize the random phenomena with a periodic nature.

We consider the following semi-linear SDE with multiplicative noise
\begin{equation}\label{eq:SDE}
	\text{d} X_{t}^{t_0}
	=
	\big(-AX_{t}^{t_0} + f(t,X_{t}^{t_0})\big)\diff{t}
	+
	g(t,X^{t_0}_{t})\diff{W_{t}},
	\quad t > t_{0}
\end{equation}
where $X_{t_{0}}^{t_{0}} = \xi$ is a $\R^{d}$-valued random variable, $f \colon \R \times \R^{d} \to \R^{d}$, $g \colon \R \times \R^{d} \to \R^{d \times m}$ are continuous functions and $A$ is a positive definite $d \times d$ symmetric matrix. Besides, $\{W_{t}\}_{t \in \R}$ is a standard two-sided Wiener process in $\R^{m}$ on a probability space $(\Omega,\F,\P)$. Here the notation $X_{t}^{t_0}, t \geq t_{0}$ denotes the solution process starting from $t_{0}$. Under Assumption \ref{asm:L} specified later, \eqref{eq:SDE} admits a unique solution $\{X_{t}^{t_0}\}_{t \geq t_{0}}$ (we also rewrite it as $\{X_{t}^{t_0}(\xi)\}_{t \geq t_{0}}$ to emphasize its dependence on the initial value $\xi$), described by the variation of constant formula as follows
\begin{align}
	X_{t}^{t_{0}}(\xi)
	=
	e^{-A(t-t_{0})}\xi
	+
	\int_{t_{0}}^{t} e^{-A(t-s)} f(s,X_{s}^{t_{0}}) \diff{s}
	+
	\int_{t_{0}}^{t} e^{-A(t-s)} g(s,X_{s}^{t_{0}}) \diff{W_{s}},
	\quad t \geq t_{0}.
\end{align}
Then the map $u_{t_{0},t} \colon \Omega \times \R^{d} \to \R^{d}$ defined by $u_{t_{0},t}(\omega,\xi) = X_{t}^{t_{0}}(\omega,\xi)$ satisfies the semiflow property
$u_{r,t}(\omega,\xi) = u_{s,t}(\omega, u_{r,s}(\omega,\xi))$
and the periodic property (with period $\tau > 0$) {$u_{s+\tau,t+\tau}(\omega,\xi) = u_{s,t}(\Theta_{\tau}\omega,\xi)$} with $\Theta_{t}(\omega)(s) := W_{t+s}-W_{t}$ for all {$r \leq s \leq t, \omega \in \Omega$ and $\xi \in \R^{d}$}.
Theorem \ref{thm:ep} shows that the pull-back sequence $\{X_{t}^{-k\tau}(\xi)\}_{k \in \N}$ has a limit $X_{t}^{*}$ in $L^{2}(\Omega;\R^{d})$, where $X_{t}^{*}$ is the random periodic solution of \eqref{eq:SDE} and is given by
\begin{align}
	X_{t}^{*}
	=
	\int_{-\infty}^{t} e^{-A(t-s)} f(s,X_{s}^{*}) \diff{s}
	+
	\int_{-\infty}^{t} e^{-A(t-s)} g(s,X_{s}^{*}) \diff{W_{s}}.
\end{align}
Then the numerical approximation of the random periodic solution $X_{t}^{*}$ becomes a numerical approximation of an infinite time horizon problem.

Concerning the numerical approximation of random periodic solution, the seminal work \cite{feng2017numerical} establishes the standard mean square convergence orders of the Euler--Maruyama method and a modified Milstein method for dissipative SDEs with multiplicative noise under the global Lipschitz conditions.
Subsequently, \cite{rong2020numerical} generalizes the Euler--Maruyama method to the ST methods and derives the convergence order $1/4$.
Taking into account that most SDEs arising from applications possess super-linearly or sub-linearly growing coefficients, it is necessary to explore the numerical results for SDEs beyond the global Lipschitz conditions. For SDEs with additive noise under the one-sided Lipschitz condition, \cite{wu2023backward} proves that the random periodic solution of the backward Euler--Maruyama method converges strongly in the mean square sense to that of SDEs with order $1/2$, which has been lifted to order $1$ under a more relaxed condition in \cite{guo2023order}. However, to the best of our knowledge, there are not any results about the numerical approximation of random periodic solution for SDEs with multiplicative noise under non-globally Lipschitz conditions, which motivates us to make a contribution to this problem.

Noting that we are dealing with a numerical approximation of an infinite time horizon problem, we take a stepsize $\Delta \in (0,1)$ and define an equidistant partition $\mathcal{P}^{\Delta} := \{t_{j}:=j\Delta\}_{j \in \Z}$ for the real line $\R$. Throughout this work, we always let the parameter $\theta \in (1/2,1]$. In order to numerically approximate the solution process $\{X_{t}^{-k\tau}\}_{t \geq -k\tau}$ of \eqref{eq:SDE} starting at $-k\tau$ with $k \in \N$, the ST methods are given by
\begin{align}\label{eq:BEM}
	\notag
	\hat{X}_{-k\tau+t_{j+1}}^{-k\tau}
	=&~
	\hat{X}_{-k\tau+t_{j}}^{-k\tau}
	+
	\theta\Delta \big( -A \hat{X}_{-k\tau+t_{j+1}}^{-k\tau}
	+
	f(t_{j+1},\hat{X}_{-k\tau+t_{j+1}}^{-k\tau}) \big)
	\\&~
	+
	(1-\theta)\Delta \big(-A \hat{X}_{-k\tau+t_{j}}^{-k\tau}
	+
	f(t_{j},\hat{X}_{-k\tau+t_{j}}^{-k\tau}) \big)
	+
	g(t_{j},\hat{X}_{-k\tau+t_{j}}^{-k\tau})
	\Delta W_{-k\tau+t_{j}}
\end{align}
for all {$j \in \N$}, where $\hat{X}_{-k\tau}^{-k\tau} = \xi$ is the initial value and $\Delta W_{-k\tau+t_{j}} := W_{-k\tau+t_{j+1}} - W_{-k\tau+t_{j}}$ is the Brownian increment. Based on the uniform bounds for the second moment of the numerical approximation { $\{\hat{X}_{-k\tau+t_{j}}^{-k\tau}\}_{j \in \N}$}, we employ Lemma \ref{lem:numrob} to show that each ST method admits a unique random periodic solution $\hat{X}_{t}^{*}$. Without resorting to the discrete Gronwall inequality, a technical iterative argument is introduced in the error estimate to overcome the difficulties arising by the superlinear growth of coefficients and the unboundedness of the time intervals. With the help of the H\"{o}lder continuity of the exact solution, we derive a uniform bound for the error between $\hat{X}_{-k\tau+t_{j}}^{-k\tau}$ and $X_{-k\tau+t_{j}}^{-k\tau}$ for all $k,j \in \N$, and consequently establish the mean square convergence of random periodic solutions.
To be precise, the mean square convergence order of each ST method is $1/2$ for SDEs with multiplicative noise and $1$ for SDEs with additive noise, which, in particular, coincides with \cite[Theorem 4.6]{guo2023order} investigating ST method with $\theta = 1$ for SDEs with additive noise.

The remainder of this paper is organized as follows.
Section \ref{sec:assumption} contains some assumptions and the existence and uniqueness of random periodic solutions of SDEs \eqref{eq:SDE}.
In Section \ref{sec:solutionST}, we present the well-posedness and the existence of unique random periodic solutions of the ST methods. The mean square convergence for random periodic solutions of the ST methods is established in Section \ref{eq:convergence}. Finally, several numerical experiments are given in Section \ref{sec:numerical} to illustrate the theoretical results.

\section{Assumptions and random periodic solutions of SDEs}\label{sec:assumption}

We start with some notation that will be used throughout the paper. Let $\N := \{0,1,2,\cdots\}$ be the set of nonnegative integers and $\Z := \{0, \pm1, \pm2, \cdots\}$ the set of all integers.
Let $|\cdot|$ and $\langle \cdot,\cdot \rangle$ denote the Euclidean norm and the corresponding inner product of vector in $\R^{d}$. By $A^{\top}$ we denote the transpose of a vector or a matrix $A$. For any matrix $A$, we use $|A| := \sqrt{\tr{(A^{\top}A)}}$ to denote its trace norm. Let $\{W_{t}\}_{t \in \R}$ be a standard two-sided Wiener process in $\R^{m}$ on a probability space $(\Omega,\F,\P)$, i.e., $W_{t} := W_{t}^{1}, t \geq 0$ and $W_{t} := W_{-t}^{2}, t < 0$ for two independent Brownian motions $\{W_{t}^{1}\}_{t \geq 0}$ and $\{W_{t}^{2}\}_{t \geq 0}$. The filtration $\{\F_{t}\}_{t \in \R}$ is defined by $\F_{t}^{s} := \sigma\{W_{u}-W_{v} : s \leq v \leq u \leq t\}$ for $-\infty < s \leq t < \infty$ and $\F_{t} := \F_{t}^{-\infty} = \bigvee_{s \leq t} \F_{t}^{s}$ for all $t \in \R$. Denote the standard $\P$-preserving ergodic Wiener shift by $\Theta \colon \R \times \Omega \to \Omega$, $\Theta_{t}(\omega)(s) := W_{t+s}-W_{t}, t,s \in \R$. For any $p \geq 1$, we denote by $(L^{p}(\Omega;\R^{d}),|\cdot|_{L^{p}(\Omega;\R^{d})})$ the Banach space of all $\R^{d}$-valued random variables $\eta$ with $|\eta|_{L^{p}(\Omega;\R^{d})} := \big(\E[|\eta|^{p}]\big)^{\frac{1}{p}} < \infty$. For simplicity, the letter $C$ stands for a generic positive constant whose value may vary for each appearance, but independent of the time interval and the stepsize $\Delta$ of the considered numerical method. At last, we refer, e.g., \cite{feng2011pathwise, zhao2009random, feng2016anticipating, feng2017numerical} for more details on the definition of the random periodic solution for stochastic semiflows.

In order to provide efficient numerical approximation for SDEs \eqref{eq:SDE}, we put the following conditions on  \eqref{eq:SDE}.

\begin{assumption}\label{asm:L}
	Suppose the following conditions are satisfied.
	
	\noindent {\rm{(A1)}}
	The eigenvalues $\{\lambda_i\}_{i=1,2,\cdots,d}$ of the symmetric matrix $A$ satisfy $\lambda := \min\{ \lambda_i\}_{i=1,2,\cdots,d} > 0.$
	
	\noindent {\rm{(A2)}}
	The coefficients $f$ and $g$ are continuous and periodic in time with period $\tau > 0$, i.e.,
	\begin{equation}\label{asm:fg-tau}
		f(t+\tau,x)=f(t,x), \quad g(t+\tau,x)=g(t,x),
		\quad x\in\mathbb{R}^{d},~ t\in\mathbb{R}.
	\end{equation}
	
	\noindent {\rm{(A3)}}
	There exist constants $\gamma \geq 1$, {$p^{*} > (d+4)\gamma$,} $L_{f} \in (0,\lambda)$ and $C > 0$ such that for all $x,y \in \R^d$ and $t, s \in \R$,
	{\begin{equation}\label{asm:fg-fg}
			\langle x-y,f(t,x)-f(s,y) \rangle
			+
			(p^{*}-1)|g(t,x)-g(s,y)|^{2}
			\leq
			L_{f}|x-y|^2 + C(1+|x|+|y|)^{\gamma+1}|t-s|
	\end{equation}}
	and
	\begin{equation}\label{eq:fpolynomial}
		|f(t,x)-f(s,y)|
		\leq
		C\big((1+|x|+|y|)^{\gamma-1} |x-y|
		+(1+|x|+|y|)^{\gamma} |t-s| \big).
	\end{equation}
	
	\noindent {\rm{(A4)}}
	For any { $0< p \leq p^*$, } there exists a constant $C > 0$ such that $|\xi|_{L^{p}(\Omega;\R^{d})} \leq C$.

	\noindent {\rm{(A5)}}
    Let $\Pi(s,x)$ be the explosion time of the solution $X_{t}^{s}(x)$ to \eqref{eq:SDE} with initial value $x\in\R^d$ and satisfy
    $\P(\Pi(s,x)=\infty, s \in \R, x \in \R^{d}) = 1$.

\end{assumption}

	For later convenience, we note several consequences of Assumption \ref{asm:L}. \rm{(A1)} implies
	\begin{align}\label{eq:Ainequality}
		\big\langle -Ax,x \big\rangle
		\leq
		-\lambda|x|^{2},
		\quad x \in \R^{d}.
	\end{align}
	From \eqref{asm:fg-fg}, we obtain
	\begin{equation}\label{eq:fgmonotonicity}
		\langle x-y,f(t,x)-f(t,y) \rangle
		+
		(p^{*}-1)|g(t,x)-g(t,y)|^{2}
		\leq
		L_{f}|x-y|^2, \quad x,y \in\mathbb{R}^d, t \in \R,
	\end{equation}
	which together with the weighted Young inequality shows that for any $x \in \R^{d}, t \in \R$,
	\begin{align}\label{eq:monocondition}
		&~\langle x,f(t,x) \rangle + \frac{p^{*}-1}{2}|g(t,x)|^{2}
		\notag
		\\\leq&~
		\langle x,f(t,x)-f(t,0) \rangle
		+
		(p^{*}-1)|g(t,x)-g(t,0)|^{2}
		+
		\langle x,f(t,0) \rangle
		+
		(p^{*}-1)|g(t,0)|^{2}
		\notag
		\\\leq&~
		L_{f}|x|^{2} + \frac{\lambda-L_{f}}{2}|x|^{2}
		+
		\frac{1}{2(\lambda-L_{f})}|f(t,0)|^{2}
		+
		(p^{*}-1)|g(t,0)|^{2}
		\\\leq&~
		\widetilde{L_{f}}|x|^{2} + \widetilde{C}
		\notag
	\end{align}
	with
	\begin{align*}
		\widetilde{L_{f}}
		:=
		\frac{\lambda+L_{f}}{2} \in \big(L_{f},\lambda\big),
		\quad
		\widetilde{C}
		:=
		\frac{1}{2(\lambda-L_{f})}
		\bigg(\sup_{t \in [0,\tau)}|f(t,0)|^{2}\bigg)
		+
		(p^{*}-1)\bigg(\sup_{t \in [0,\tau)}|g(t,0)|^{2}\bigg) > 0.
	\end{align*}
	Besides, combining \eqref{asm:fg-fg} and \eqref{eq:fpolynomial} shows that there exists a constant $C > 0$ such that
	\begin{align}\label{eq:gpolynomial}
		|g(t,x)-g(s,y)|^{2}
		\leq
		C(1+|x|+|y|)^{\gamma-1}|x-y|^2
		+
		C(1+|x|+|y|)^{\gamma+1}|t-s|
	\end{align}
	for all $x,y \in \R^d$ and $t, s \in \R$. Together with \eqref{eq:fpolynomial}, we derive
	\begin{align}\label{eq:fgsuperlinear}
		|f(t,x)| \leq C(1+|x|)^{\gamma},
		\quad
		|g(t,x)|^{2} \leq C(1+|x|)^{\gamma+1},
		\quad x \in \R^{d}, t \in \R.
	\end{align}

According to \eqref{eq:fpolynomial} and \eqref{eq:gpolynomial}, we know that $f$ and $g$ are locally Lipschitz continuous on $[t_{0},T] \times \R^{d}$, where $[t_{0},T]$ is a finite subinterval of $[t_{0},\infty)$ for any $T > t_{0}$. In combination with the monotone condition \eqref{eq:monocondition}, \cite[Theorem 3.5 in Chapter 2]{mao2008stochastic} ensures that \eqref{eq:SDE} admits a unique strong solution $\{X_{t}^{t_{0}}\}_{t \in [t_{0},T]}$, given by
	\begin{align}
		X_{t}^{t_{0}}
		=
		\xi + \int_{t_{0}}^{t} -AX_{s}^{t_{0}}
		+ f(s,X_{s}^{t_{0}}) \diff{s}
		+ \int_{t_{0}}^{t} g(s,X_{s}^{t_{0}}) \diff{W_{s}},
		\quad t \in [t_{0},T].
	\end{align}
      Due to the arbitrariness of $[t_{0},T] \subset [t_{0},\infty)$, \eqref{eq:SDE} possesses a unique global solution $\{X_{t}^{t_{0}}\}_{t \geq t_{0}}$.
%
%
\begin{equation}\label{eq:ustmap}
      u_{s,t} \colon \Omega \times \R^{d} \to \R^{d}, (\omega,x) \mapsto u_{s,t}(\omega,x) := X_{t}^{s}(\omega,x)
\end{equation}
satisfies the semiflow property
\begin{equation}\label{eq:semiflow}
      u_{s,t}(x)
      =
      u_{r,t}(u_{s,r}(x)),
      \quad
      s \leq r \leq t,
      ~x \in \R^{d},
      ~\text{almost surely}.
\end{equation}


To establish the existence and uniqueness of random periodic solution of \eqref{eq:SDE}, several auxiliary lemmas are stated and proved below. The first one considers the uniform boundedness for the exact solution of \eqref{eq:SDE}.

	\begin{lemma}\label{lem:exactbound}
		Suppose that Assumption \ref{asm:L} holds. Then for any $p \in [2,p^{*}]$, there exists a constant $C > 0$ such that
		\begin{align}\label{eq:exactbound}
			\sup_{t_0 \in \mathbb{R}} \sup_{t \geq t_0}
			\E\big[|X_t^{t_0}|^{p}\big]
			\leq
			C.
		\end{align}
	\end{lemma}

	\begin{proof}
	Applying the It\^{o} formula shows that for any $t_0 \in \mathbb{R}$, $t \geq t_0$ and $p \in (2,p^{*}]$,
	\begin{align}
		e^{p\lambda t}|X_t^{t_0}|^{p}
		=&~
		e^{p \lambda t_0}|\xi|^{p}
		+
		\int_{t_0}^{t}
		p \lambda e^{p\lambda s} |X_{s}^{t_0}|^{p}
		+
		p e^{p\lambda s} |X_{s}^{t_0}|^{p-2}
		\big\langle X_{s}^{t_0}, -AX_{s}^{t_0}\big\rangle
		\notag
		\\&~
		+ p e^{p\lambda s} |X_{s}^{t_0}|^{p-2}
		\big\langle X_{s}^{t_0}, f(s,X_{s}^{t_0})\big\rangle
		+
		\frac{p}{2} e^{p\lambda s} |X_{s}^{t_0}|^{p-2}
		|g(s,X_{s}^{t_0})|^{2}
		\notag
		\\&~
		+ \frac{p(p-2)}{2} e^{p\lambda s}
		|X_{s}^{t_0}|^{p-4}
		|(X_{s}^{t_0})^{\top}g(s,X_{s}^{t_0})|^{2}
		\diff{s}
		\notag
		\\&~
		+
		\int_{t_0}^{t} p e^{p\lambda s}
		|X_{s}^{t_0}|^{p-2}
		\big\langle X_{s}^{t_0},
		g(s,X_{s}^{t_0}) \diff{W_{s}}\big\rangle.
		\notag
	\end{align}
	It follows from the Cauchy--Schwarz inequality, \eqref{eq:Ainequality} and \eqref{eq:monocondition} that
	\begin{align}
		&~\E\big[e^{p\lambda t}|X_t^{t_0}|^{p}\big]
		\notag
		\\\leq&~
		\E\big[e^{p\lambda t_0}|\xi|^{p}\big]
		+
		\int_{t_0}^{t} pe^{p\lambda s}
		\E\big[|X_{s}^{t_0}|^{p-2}
		\big(\lambda |X_{s}^{t_0}|^{2}
		+
		\big\langle X_{s}^{t_0},
		-AX_{s}^{t_0}\big\rangle \big)\big] \diff{s}
		\notag
		\\&~
		+
		\int_{t_0}^{t} p e^{p\lambda s}
		\E\Big[|X_{s}^{t_0}|^{p-2}
		\Big(\big\langle X_{s}^{t_0}, f(s,X_{s}^{t_0})\big\rangle
		+
		\frac{p-1}{2}
		|g(s,X_{s}^{t_0})|^{2}\Big)\Big]
		\diff{s}
		\notag
		\\\leq&~
		\E\big[e^{p\lambda t_0} |\xi|^{p}\big]
		+
		\int_{t_0}^{t}
		p e^{p\lambda s} \big(\widetilde{L_{f}}
		\E\big[|X_{s}^{t_0}|^{p}\big]
		+
		\widetilde{C}
		\E\big[|X_{s}^{t_0}|^{p-2}\big]\big)
		\diff{s}.
		\notag
	\end{align}
	Due to the Young inequality $|x|^{p-2}|y| \leq \frac{p-2}{p}|x|^{p} + \frac{2}{p}|y|^{\frac{p}{2}}$ for any $x,y \in \R$, we choose a fixed constant $\varepsilon \in (0,\lambda-\widetilde{L_{f}})$ to derive
	\begin{align*}
		\widetilde{C}|X_{s}^{t_0}|^{p-2}
		=&~
		\frac{p}{p-2} \times
		\bigg( (\lambda-\widetilde{L_{f}}
		-\varepsilon)^{\frac{p-2}{p}}
		|X_{s}^{t_0}|^{p-2}
		\times
		\frac{p-2}{p}
		(\lambda-\widetilde{L_{f}}-\varepsilon)^{\frac{2-p}{p}}
		\widetilde{C}\bigg)
		\\\leq&~
		\frac{p}{p-2} \times
		\bigg( \frac{p-2}{p}
		(\lambda-\widetilde{L_{f}}-\varepsilon)
		|X_{s}^{t_0}|^{p}
		+
		\frac{2}{p}\Big(\frac{p-2}{p}
		(\lambda-\widetilde{L_{f}}-\varepsilon)^{\frac{2-p}{p}}
		\widetilde{C}\Big)^{\frac{p}{2}}\bigg)
		\\\leq&~
		(\lambda-\widetilde{L_{f}}-\varepsilon)
		|X_{s}^{t_0}|^{p}
		+
		\frac{2}{p-2}\big(
		(\lambda-\widetilde{L_{f}}-\varepsilon)^{\frac{2-p}{p}}
		\widetilde{C}\big)^{\frac{p}{2}},
	\end{align*}
	and consequently
	\begin{align}
		\E\big[e^{p\lambda t}|X_t^{t_0}|^{p}\big]
		\leq
		\E\big[ e^{p\lambda t_0} |\xi|^{p}\big]
		+
		Ce^{p\lambda t}
		+
		\int_{t_0}^{t} p(\lambda-\varepsilon)
		\E\big[e^{p\lambda s}
		|X_{s}^{t_0}|^{p}\big] \diff{s}.
		\notag
	\end{align}
	The Gronwall inequality \cite[Theorem 1]{dragomir2003gronwall} implies
	\begin{align}
		\E\big[e^{p\lambda t}|X_t^{t_0}|^{p}\big]
		\leq&~
		\E\big[e^{p\lambda t_0} |\xi|^{p}\big]
		+
		Ce^{p\lambda t}
		+
		\int_{t_0}^{t} p(\lambda-\varepsilon)
		\big(\E\big[ e^{p\lambda t_0} |\xi|^{p}\big]
		+
		Ce^{p\lambda s}\big)
		e^{\int_{s}^{t} p(\lambda-\varepsilon) \diff{u}}
		\diff{s}
		\notag
		\\\leq&~
		\E\big[ e^{p\lambda t_0} |\xi|^{p}\big]
		+
		Ce^{p\lambda t}
		+
		\E\big[ e^{p\lambda t_0} |\xi|^{p}\big]e^{p(\lambda-\varepsilon)(t-t_0)}
		+
		C\frac{\lambda-\varepsilon}{\varepsilon}e^{p\lambda t}.
		\notag
	\end{align}
	Then we obtain \eqref{eq:exactbound} for $p \in (2,p^{*}]$, which together with the H\"{o}lder inequality indicates \eqref{eq:exactbound} for $p =2$. Thus we finish the proof.
\end{proof}

	The forthcoming lemma explores the dependence of solutions of \eqref{eq:SDE} on different initial values.
	
	\begin{lemma}\label{lem:rob}
		Let $X_t^{t_0}(\xi)$ and $X_t^{t_0}(\eta)$ be two solutions of  \eqref{eq:SDE} with different initial values $\xi$ and $\eta$, respectively. Suppose that Assumption \ref{asm:L} holds for both initial values $\xi$ and $\eta$. Then for any $p \in [2,p^{*}], t_0 \in \R$ and $t \geq t_0$,
		\begin{equation}\label{lem:rob:rs}
			\E\big[|X_t^{t_0}(\xi)-X_t^{t_0}(\eta)|^{p}\big]
			\leq
			e^{p(L_{f}-\lambda)(t-t_0)}
			\E\big[|\xi-\eta|^{p}\big].
		\end{equation}
	\end{lemma}

	\begin{proof}
	Because of $X_{t_0}^{t_0}(\xi)-X_{t_0}^{t_0}(\eta) = \xi-\eta$ and
	\begin{align}
		\notag
		\diff{(X_{t}^{t_0}(\xi)-X_{t}^{t_0}(\eta))}
		=&~
		\big(-A(X_{t}^{t_0}(\xi)-X_{t}^{t_0}(\eta))
		+
		(f(t,X_{t}^{t_0}(\xi))-f(t,X_{t}^{t_0}(\eta)))
		\big) \diff{t}
		\\&~\notag
		+
		\big(g(t,X_{t}^{t_0}(\xi))
		-g(t,X_{t}^{t_0}(\eta))\big)\diff{W_{t}},
		\quad t > t_0,
	\end{align}
	we use the It\^{o} formula to get
	\begin{align}
		\notag
		&~e^{p\lambda t}|X_{t}^{t_0}(\xi)-X_{t}^{t_0}(\eta)|^{p}
		\\=&~\notag
		e^{p\lambda t_0} |\xi-\eta|^{p}
		+
		\int_{t_0}^{t}
		p\lambda e^{p\lambda s}
		|X_{s}^{t_0}(\xi)-X_{s}^{t_0}(\eta)|^{p}
		\\&~+\notag
		pe^{p\lambda s}
		|X_{s}^{t_0}(\xi)-X_{s}^{t_0}(\eta)|^{p-2}
		\big\langle X_{s}^{t_0}(\xi)-X_{s}^{t_0}(\eta),
		-A(X_{s}^{t_0}(\xi)-X_{s}^{t_0}(\eta)) \big\rangle
		\diff{s}
		\\&~+\notag
		\int_{t_0}^{t}
		p e^{p\lambda s}
		|X_{s}^{t_0}(\xi)-X_{s}^{t_0}(\eta)|^{p-2}
		\big\langle X_{s}^{t_0}(\xi)-X_{s}^{t_0}(\eta),
		f(s,X_{s}^{t_0}(\xi))-f(s,X_{s}^{t_0}(\eta)) \big\rangle
		\\&~+\notag
		\frac{p}{2}e^{p\lambda s}
		|X_{s}^{t_0}(\xi)-X_{s}^{t_0}(\eta)|^{p-2}
		|g(s,X_{s}^{t_0}(\xi))-g(s,X_{s}^{t_0}(\eta))|^{2}
		\\&~+\notag
		\frac{p(p-2)}{2} e^{p\lambda s}
		|X_{s}^{t_0}(\xi)-X_{s}^{t_0}(\eta)|^{p-4}
		|(X_{s}^{t_0}(\xi)-X_{s}^{t_0}(\eta))^{\top}
		(g(s,X_{s}^{t_0}(\xi))-g(s,X_{s}^{t_0}(\eta)))|^{2}
		\diff{s}
		\\&~+\notag
		\int_{t_0}^{t} pe^{p\lambda s}
		|X_{s}^{t_0}(\xi)-X_{s}^{t_0}(\eta)|^{p-2}
		\big\langle X_{s}^{t_0}(\xi)-X_{s}^{t_0}(\eta),
		(g(s,X_{s}^{t_0}(\xi))-g(s,X_{s}^{t_0}(\eta)))
		\diff{W_{s}} \big\rangle.
	\end{align}
	Taking expectations and utilizing the Cauchy--Schwarz inequality, \eqref{eq:Ainequality} as well as \eqref{eq:fgmonotonicity} show that for any $p \in [2,p^{*}]$, 
	\begin{align}
		\notag
		\E\big[e^{p\lambda t}
		|X_{t}^{t_0}(\xi)-X_{t}^{t_0}(\eta)|^{p}\big]
		\leq&~
		\E\big[e^{p\lambda t_0} |\xi-\eta|^{p}\big]
		+
		\int_{t_0}^{t} p
		\E\Big[e^{p\lambda s}
		|X_{s}^{t_0}(\xi)-X_{s}^{t_0}(\eta)|^{p-2}
		\\&~\times\notag
		\Big(\big\langle X_{s}^{t_0}(\xi)-X_{s}^{t_0}(\eta),
		f(s,X_{s}^{t_0}(\xi))-f(s,X_{s}^{t_0}(\eta)) \big\rangle
		\\&~+\notag
		\frac{p-1}{2} |g(s,X_{s}^{t_0}(\xi))
		-g(s,X_{s}^{t_0}(\eta))|^{2}\Big)\Big] \diff{s}
		\\\leq&~\notag
		\E\big[e^{p\lambda t_0} |\xi-\eta|^{p}\big]
		+
		\int_{t_0}^{t} p L_{f}
		\E\big[e^{p\lambda s}|X_{s}^{t_0}(\xi)
		- X_{s}^{t_0}(\eta)|^{p}\big]\diff{s}.
	\end{align}
	It follows from the Gronwall inequality
	\cite[Corollary 3]{dragomir2003gronwall} that
	\begin{align}
		\notag
		\E\big[e^{p\lambda t}
		|X_{t}^{t_0}(\xi)-X_{t}^{t_0}(\eta)|^{p}\big]
		\leq&~
		\E\big[e^{p\lambda t_0} |\xi-\eta|^{p}\big]
		e^{pL_{f}(t-t_0)},
	\end{align}
	which immediately results in \eqref{lem:rob:rs}.
\end{proof}

The forthcoming result presents the H\"{o}lder continuity of the exact solution with respect to the final time.

\begin{lemma}\label{lem:exactholder}
	Suppose that Assumption \ref{asm:L} holds. Then for any $p \in [2,\frac{p^{*}}{\gamma}]$, $t_0 \in \R$ and $s,t \geq t_0 $, there exists a constant $C > 0$, independent of $t_0, s, t$, such that
	\begin{align}\label{eq:exactholder}
		\E\big[|X_{t}^{t_0 } - X_{s}^{t_0 }|^{p}\big]
		\leq
		C\big(|t-s|^{\frac{p}{2}} + |t-s|^{p}\big).
	\end{align}
\end{lemma}

\begin{proof}
Without loss of generality, we assume $s \leq t$ and get
\begin{align}\label{lem:X:1}
	\notag
	&~\E\big[|X_{t}^{t_0} - X_{s}^{t_0}|^{p}\big]
	\\=&~
	\E\bigg[\Big|\int_{s}^{t}
	-AX_{r}^{t_0} + f(r,X_{r}^{t_0}) \diff{r}
	+
	\int_{s}^{t} g(r,X_{r}^{t_0}) \diff{W_{r}}\Big|^{p}\bigg]
	\\\leq&~\notag
	C\E\bigg[\Big|\int_{s}^{t}
	-AX_{r}^{t_0} + f(r,X_{r}^{t_0})
	\diff{r}\Big|^{p}\bigg]
	+
	C \E\bigg[\Big|\int_{s}^{t} g(r,X_{r}^{t_0}) \diff{W_{r}}\Big|^{p}\bigg].
\end{align}
By the H\"older inequality, \eqref{eq:fgsuperlinear} and Lemma \ref{lem:exactbound}, we have
\begin{align}\label{lem:X:f}
	&~\notag
	\E\bigg[\Big|\int_{s}^{t}
	-AX_{r}^{t_0} + f(r,X_{r}^{t_0})
	\diff{r}\Big|^{p}\bigg]
	\\\leq&~\notag
	(t-s)^{p-1}\int_{s}^{t}\E\big[|
	-AX_{r}^{t_0} + f(r,X_{r}^{t_0})
	|^{p}\big]\diff{r}
	\\\leq&~
	C(t-s)^{p-1}\int_{s}^{t}
	\E\big[|X_{r}^{t_0}|^{p}\big]
	+
	\E\big[|f(r,X_{r}^{t_0})|^{p}\big]\diff{r}
	\\\leq&~\notag
	C(t-s)^{p-1}\int_{s}^{t}
	1 + \E\big[|X_{r}^{t_0}|^{p}\big]
	+
	\E\big[|X_{r}^{t_0}|^{p\gamma}\big]\diff{r}
	\\\leq&~\notag
	C(t-s)^{p}.
\end{align}
Applying \cite[Theorem 7.1 in Chapter 1]{mao2008stochastic}, \eqref{eq:fgsuperlinear} and Lemma \ref{lem:exactbound} yields
\begin{align}\label{lem:X:s}
	\E\bigg[\Big|\int_{s}^{t}
	g(r,X_{r}^{t_0}) \diff{W_{r}}\Big|^{p}\bigg]
	\leq&~\notag
	C(t-s)^{\frac{p}{2}-1}\int_{s}^{t}
	\E\big[|g(r,X_{r}^{t_0})|^{p}\big] \diff{r}
	\\\leq&~
	C(t-s)^{\frac{p}{2}-1}\int_{s}^{t}
	1 + \E\big[|X_{r}^{t_0}
	|^{\frac{p(\gamma+1)}{2}}\big] \diff{r}
	\\\leq&~\notag
	C(t-s)^{\frac{p}{2}}.
\end{align}
Substituting \eqref{lem:X:f} and \eqref{lem:X:s} into \eqref{lem:X:1} enables us to derive \eqref{eq:exactholder}.
\end{proof}

Furthermore, one can show the H\"{o}lder continuity of the exact solution with respect to the initial time.
\begin{lemma}\label{lem:exactholderinitial}
      Suppose that Assumption \ref{asm:L} holds. Then for any $p \in [2,\frac{p^{*}}{\gamma}]$, $t \geq t_{0}$ and $t \geq s_{0}$, there exists a constant $C > 0$, independent of $t, s_{0}, t_{0}$, such that
      \begin{align*}
            \E\big[|X_{t}^{s_{0}}(\xi) - X_{t}^{t_{0}}(\xi)|^{p}\big]
            \leq
            C(|t_{0}-s_{0}|^{\frac{p}{2}} + |t_{0}-s_{0}|^{p}).
      \end{align*}
\end{lemma}

\begin{proof}
      Without loss of generality, we assume $s_{0} \leq t_{0}$ and use \eqref{eq:semiflow} to get
      \begin{align}
            X_{t}^{s_{0}}(\xi) - X_{t}^{t_{0}}(\xi)
            =
            u_{s_{0},t}(\xi) - u_{t_{0},t}(\xi)
            =
            u_{t_{0},t}(u_{s_{0},t_{0}}(\xi)) - u_{t_{0},t}(\xi)
            =
            X_{t}^{t_{0}}(X_{t_{0}}^{s_{0}}(\xi))
            - X_{t}^{t_{0}}(\xi)
      \end{align}
      almost surely. It follows from Lemmas \ref{lem:exactbound} and \ref{lem:rob}, $L_{f} \in (0,\lambda)$ and $t_{0} \leq t$ that
      \begin{align*}
            \E\big[|X_{t}^{s_{0}}(\xi) - X_{t}^{t_{0}}(\xi)|^{p}\big]
            =&~
            \E\big[|X_{t}^{t_{0}}(X_{t_{0}}^{s_{0}}(\xi))
            - X_{t}^{t_{0}}(\xi)|^{p}\big]
            \\\leq&~
            e^{p(L_{f}-\lambda)(t-t_{0})}
            \E\big[|X_{t_{0}}^{s_{0}}(\xi)
            - \xi|^{p}\big]
            \\\leq&~
            \E\big[|X_{t_{0}}^{s_{0}}(\xi)
            - X_{s_{0}}^{s_{0}}(\xi)|^{p}\big],
      \end{align*}
      which together with Lemma \ref{lem:exactholder} finishes the proof.
\end{proof}

	Now we can prove the existence and uniqueness of random periodic solution for \eqref{eq:SDE}.

	\begin{theorem}\label{thm:ep}
		Suppose that Assumption \ref{asm:L} holds. Then for any initial value $\xi$ satisfying {\rm{(A4)}}, the considered equation \eqref{eq:SDE} admits a unique random periodic solution {$X_{t}^{*} \in L^{2}(\Omega;\mathbb{R}^{d}),t \geq 0$}, given by
		\begin{align}
			\notag
			X_{t}^{*}
			=
			\int_{-\infty}^{t}
			e^{-A(t-s)}f(s,X_{s}^{*})\diff{s}
			+
			\int_{-\infty}^{t}
			e^{-A(t-s)}g(s,X_{s}^{*})\diff{W_{s}},
			\quad t \geq 0
		\end{align}
		such that
		\begin{equation}\label{thm:ep:rs}
			\lim_{k \rightarrow \infty}
			|X_{t}^{-k\tau} - X_{t}^{*}|_{L^{2}(\Omega;\R^{d})}
			= 0.
		\end{equation}
	\end{theorem}

\begin{proof}
Noting that \eqref{eq:SDE} admits a unique global solution,
we will first follow the procedures presented in the proof of \cite[Theorem2.5 and Proposition 5.2]{scheutzow2017strong} to attest that \eqref{eq:SDE} admits a global semiflow.
Actually, applying Lemmas \ref{lem:rob}, \ref{lem:exactholder} and \ref{lem:exactholderinitial} shows that for any $T_1,T_2 \in \R$ with $T_1<T_2$ and $p \in [2,\frac{p^{*}}{\gamma}]$, there exists $C(T_1,T_2) > 0$  such that for any $x,y \in \R^{d}$ and $s_{0},s, t_{0}, t \in [T_1, T_2]$ with $s_{0} \leq s, t_{0} \leq t$,
      \begin{align}
            \E\big[|X_{s}^{s_{0}}(x) - X_{t}^{t_{0}}(y)|^{p}\big]
            \leq
            C(T_1,T_2) (|s_{0}-t_{0}|^{\frac{p}{2}}
            + |s-t|^{\frac{p}{2}} + |x-y|^{p}).
      \end{align}
      Then $\frac{4+d}{p}<1 $ and the Kolmogorov continuity theorem \cite[Theorem 1.4.1]{kunita1990stochastic} ensure that there exists a modification $v_{s,t}$ of $u_{s,t}$ in \eqref{eq:ustmap} such that $(s,t,x) \mapsto v_{s,t}(x)$ is continuous for all $\omega \in \Omega$.

      Let $\{r_{n}\}_{n \in \N}$ be the sequence of rational numbers in $\R$. Then for each $n \in \N$, the fact $v_{s,s}(x) = u_{s,s}(x) = x$ almost surely for all $s \in \R$ and $x \in \R^{d}$ shows that there exists $\Omega_{n} \subset \Omega$ with $\P(\Omega_{n}) = 0$ such that $v_{r_{n},r_{n}}(\omega,x) = x$ for all $x \in \R^{d}$ and $\omega \in \Omega_{n}^{c}$. Putting $\widetilde{\Omega} := \cup_{n = 1}^{\infty}\Omega_{n}$ yields $\P(\widetilde{\Omega}) = 0$ and $v_{r_{n},r_{n}}(\omega,x) = x, n \in \N$ for all $x \in \R^{d}$ and $\omega \in \widetilde{\Omega}^{c}$. Together with the continuity of $(s,x) \mapsto v_{s,s}(\omega,x)$ for all $\omega \in \Omega$, we see that $v_{s,s}(\omega,x) = x$ for all $s \in \R$, $x \in \R^{d}$ and $\omega \in \widetilde{\Omega}^{c}$. Redefining $v_{s,t}(\omega,x) := x$ for any $s,t \in \R$ with $s \leq t$, $x \in \R^{d}$ and $\omega \in \widetilde{\Omega}$ enables us to derive $v_{s,s}(\omega,x) = x$ for all $s \in \R$, $x \in \R^{d}$ and $\omega \in \Omega$.

      Observing \eqref{eq:semiflow}, the uniqueness of global solution of \eqref{eq:SDE} ensures that
      \begin{equation}\label{eq:vsemiflow}
            v_{s,t}(x)
            =
            v_{r,t}(v_{s,r}(x)),
            \quad
            s \leq r \leq t,
            ~x \in \R^{d},
            ~\text{almost surely}.
      \end{equation}
      Because the functions on the both sides of \eqref{eq:vsemiflow} are continuous with respect to $(s,r,t,x)$, there exists $\widehat{\Omega} \subset \Omega$ with $\P(\widehat{\Omega}) = 0$ such that $v_{s,t}(\omega,x) = v_{r,t}(\omega,v_{s,r}(\omega,x))$ for all $s,r,t \in \R$ with $s \leq r \leq t$, all $x \in \R^{d}$ and all $\omega \in \widehat{\Omega}^{c}$. Redefining $v_{s,t}(\omega,x) := x$ for any $s,t \in \R$ with $s \leq t$, $x \in \R^{d}$ and $\omega \in \widehat{\Omega}$ results in $v_{s,t}(\omega,x) = v_{r,t}(\omega,v_{s,r}(\omega,x))$ for all $s,r,t \in \R$ with $s \leq r \leq t$, all $x \in \R^{d}$ and all $\omega \in \Omega$.

Applying (A5) in Assumption \ref{asm:L} and  \cite[Definition 1.1]{scheutzow2017strong} indicates that $v$ is a global semiflow for \eqref{eq:SDE}.
%
%
%
The remainder of this proof is analogous to that of \cite[Theorem 2.4]{feng2017numerical} and thus omitted.
%
%
%
%
\end{proof}

\section{Random periodic solutions of ST methods}\label{sec:solutionST}

This section aims to show the existence and uniqueness of random periodic solutions for the ST methods. For this purpose, let us first verify the well-posedness of \eqref{eq:BEM}. In fact, \eqref{eq:fgmonotonicity} implies
\begin{equation*}
	\langle x-y,f(t,x)-f(t,y) \rangle
	\leq
	L_{f}|x-y|^2, \quad x,y \in\mathbb{R}^d, t \in \R.
\end{equation*}
According to the uniform monotonicity theorem
\cite[Theorem C.2]{stuart1996dynamical} and using $L_{f} \in (0,\lambda)$, we know that \eqref{eq:BEM} admits a unique solution $\{\hat{X}_{-k\tau+t_{j}}^{-k\tau}\}_{j \in \N}$
for any stepsize $\Delta \in (0,1)$. Besides, there is a uniform bound for the second moment of the numerical solution $\{\hat{X}_{-k\tau+t_{j}}^{-k\tau}\}_{j \in \N}$, as indicated below.


\begin{lemma}\label{lem:numbound}
	Suppose that Assumption \ref{asm:L} holds and that $\theta \in (\frac{1}{2},1]$. Then there exists a constant $C > 0$, independent of $j,k$ and $\Delta$, such that
	\begin{equation}\label{lem:bound:rs}
		\sup_{k,j\in \mathbb{N}}
		\E\big[|\hat{X}_{-k\tau+t_{j}}^{-k\tau}|^2\big]
		\leq C.
	\end{equation}
\end{lemma}

\begin{proof}
Noticing that \eqref{eq:BEM} implies
\begin{align}\label{eq:hatXerror}
	\notag
	&~\E\big[|\hat{X}_{-k\tau+t_{j+1}}^{-k\tau}
	-
	\theta\Delta \big( -A \hat{X}_{-k\tau+t_{j+1}}^{-k\tau}
	+
	f(t_{j+1},\hat{X}_{-k\tau+t_{j+1}}^{-k\tau}) \big)|^{2}\big]
	\\=&~
	\E\big[|\hat{X}_{-k\tau+t_{j}}^{-k\tau}
	+
	(1-\theta)\Delta \big(-A \hat{X}_{-k\tau+t_{j}}^{-k\tau}
	+
	f(t_{j},\hat{X}_{-k\tau+t_{j}}^{-k\tau}) \big)
	+
	g(t_{j},\hat{X}_{-k\tau+t_{j}}^{-k\tau})
	\Delta W_{-k\tau+t_{j}}|^{2}\big]
	\\=&~\notag
	\E\big[|\hat{X}_{-k\tau+t_{j}}^{-k\tau}
	+
	(1-\theta)\Delta \big(-A \hat{X}_{-k\tau+t_{j}}^{-k\tau}
	+
	f(t_{j},\hat{X}_{-k\tau+t_{j}}^{-k\tau}) \big)|^{2}\big]
	+
	\Delta \E\big[|g(t_{j},\hat{X}_{-k\tau+t_{j}}^{-k\tau})|^{2}\big],
\end{align}
where we have used
\begin{align*}
	\E\big[\big\langle
	\hat{X}_{-k\tau+t_{j}}^{-k\tau}
	+
	(1-\theta)\Delta \big(-A \hat{X}_{-k\tau+t_{j}}^{-k\tau}
	+
	f(t_{j},\hat{X}_{-k\tau+t_{j}}^{-k\tau}),
	g(t_{j},\hat{X}_{-k\tau+t_{j}}^{-k\tau})
	\Delta{W_{-k\tau+t_{j}}}
	\big\rangle\big]
	=
	0
\end{align*}
due to the $\F_{-k\tau+t_{j}}$-measurability of $\hat{X}_{-k\tau+t_{j}}^{-k\tau}$ and the independence between $\hat{X}_{-k\tau+t_{j}}^{-k\tau}$ and  $\Delta{W_{-k\tau+t_{j}}}$ for all $j$. On one hand, \eqref{eq:Ainequality} and \eqref{eq:monocondition} enable us to show
\begin{align*}
	&~\E\big[|\hat{X}_{-k\tau+t_{j+1}}^{-k\tau}
	-
	\theta\Delta \big( -A \hat{X}_{-k\tau+t_{j+1}}^{-k\tau}
	+
	f(t_{j+1},\hat{X}_{-k\tau+t_{j+1}}^{-k\tau}) \big)|^{2}\big]
	\\=&~
	\E\big[|\hat{X}_{-k\tau+t_{j+1}}^{-k\tau}|^{2}\big]
	+
	\theta^{2}\Delta^{2}
	\E\big[|A \hat{X}_{-k\tau+t_{j+1}}^{-k\tau}
	-
	f(t_{j+1},\hat{X}_{-k\tau+t_{j+1}}^{-k\tau})|^{2}\big]
	\\&~
	+
	2\theta\Delta\E\big[\big\langle
	\hat{X}_{-k\tau+t_{j+1}}^{-k\tau},
	A \hat{X}_{-k\tau+t_{j+1}}^{-k\tau}\big\rangle\big]
	-
	2\theta\Delta\E\big[\big\langle
	\hat{X}_{-k\tau+t_{j+1}}^{-k\tau},
	f(t_{j+1},\hat{X}_{-k\tau+t_{j+1}}^{-k\tau})
	\big\rangle\big]
	\\\geq&~
	\E\big[|\hat{X}_{-k\tau+t_{j+1}}^{-k\tau}|^{2}\big]
	+
	\theta^{2}\Delta^{2}
	\E\big[|A \hat{X}_{-k\tau+t_{j+1}}^{-k\tau}
	-
	f(t_{j+1},\hat{X}_{-k\tau+t_{j+1}}^{-k\tau})|^{2}\big]
	\\&~
	+
	2\lambda\theta\Delta
	\E\big[|\hat{X}_{-k\tau+t_{j+1}}^{-k\tau}|^{2}\big]
	+
	(p^{*}-1)\theta\Delta\E\big[|g(t_{j+1},
	\hat{X}_{-k\tau+t_{j+1}}^{-k\tau})|^{2}\big]
	\\&~-
	2\widetilde{L_{f}}\theta\Delta
	\E\big[|\hat{X}_{-k\tau+t_{j+1}}^{-k\tau}|^{2}\big]
	-
	2\widetilde{C}\theta\Delta
	\\=&~
	\big(1+2(\lambda-\widetilde{L_{f}})\theta\Delta\big)
	\E\big[|\hat{X}_{-k\tau+t_{j+1}}^{-k\tau}|^{2}\big]
	+
	(p^{*}-1)\theta\Delta\E\big[|g(t_{j+1},
	\hat{X}_{-k\tau+t_{j+1}}^{-k\tau})|^{2}\big]
	\\&~
	+
	\theta^{2}\Delta^{2}
	\E\big[|A \hat{X}_{-k\tau+t_{j+1}}^{-k\tau}
	-
	f(t_{j+1},\hat{X}_{-k\tau+t_{j+1}}^{-k\tau})|^{2}\big]
	-
	2\widetilde{C}\theta\Delta.
\end{align*}
On the other hand, using \eqref{eq:Ainequality} and \eqref{eq:monocondition} again yields
\begin{align*}
	&~\E\big[|\hat{X}_{-k\tau+t_{j}}^{-k\tau}
	+
	(1-\theta)\Delta \big(-A \hat{X}_{-k\tau+t_{j}}^{-k\tau}
	+
	f(t_{j},\hat{X}_{-k\tau+t_{j}}^{-k\tau}) \big)|^{2}\big]
	\\=&~
	\E\big[|\hat{X}_{-k\tau+t_{j}}^{-k\tau}|^{2}\big]
	+
	(1-\theta)^{2}\Delta^{2}
	\E\big[|-A \hat{X}_{-k\tau+t_{j}}^{-k\tau}
	+
	f(t_{j},\hat{X}_{-k\tau+t_{j}}^{-k\tau})|^{2}\big]
	\\&~
	+
	2(1-\theta)\Delta\E\big[\big\langle
	\hat{X}_{-k\tau+t_{j}}^{-k\tau},
	-A \hat{X}_{-k\tau+t_{j}}^{-k\tau}
	+
	f(t_{j},\hat{X}_{-k\tau+t_{j}}^{-k\tau})
	\big\rangle\big]
	\\\leq&~
	\E\big[|\hat{X}_{-k\tau+t_{j}}^{-k\tau}|^{2}\big]
	+
	(1-\theta)^{2}\Delta^{2}
	\E\big[|-A \hat{X}_{-k\tau+t_{j}}^{-k\tau}
	+
	f(t_{j},\hat{X}_{-k\tau+t_{j}}^{-k\tau})|^{2}\big]
	\\&~
	-
	2\lambda(1-\theta)\Delta
	\E\big[|\hat{X}_{-k\tau+t_{j}}^{-k\tau}|^{2}\big]
	+
	2\widetilde{L_{f}}(1-\theta)\Delta
	\E\big[|\hat{X}_{-k\tau+t_{j}}^{-k\tau}|^{2}\big]
	\\&~
	+
	2(1-\theta)\Delta\widetilde{C}
	-
	(p^{*}-1)(1-\theta)\Delta \E\big[|g(t_{j},
	\hat{X}_{-k\tau+t_{j}}^{-k\tau})|^{2}\big]
	\\=&~
	\big(1-2(\lambda-\widetilde{L_{f}})(1-\theta)\Delta\big)
	\E\big[|\hat{X}_{-k\tau+t_{j}}^{-k\tau}|^{2}\big]
	\\&~+
	(1-\theta)^{2}\Delta^{2}
	\E\big[|A \hat{X}_{-k\tau+t_{j}}^{-k\tau}
	-
	f(t_{j},\hat{X}_{-k\tau+t_{j}}^{-k\tau})|^{2}\big]
	\\&~
	+
	2(1-\theta)\Delta\widetilde{C}
	-
	(p^{*}-1)(1-\theta)\Delta \E\big[|g(t_{j},
	\hat{X}_{-k\tau+t_{j}}^{-k\tau})|^{2}\big].
\end{align*}
It follows that
\begin{align*}
	&~\big(1+2(\lambda-\widetilde{L_{f}})\theta\Delta\big)
	\E\big[|\hat{X}_{-k\tau+t_{j+1}}^{-k\tau}|^{2}\big]
	\\&~+
	(p^{*}-1)\theta\Delta\E\big[|g(t_{j+1},
	\hat{X}_{-k\tau+t_{j+1}}^{-k\tau})|^{2}\big]
	\\&~
	+
	\theta^{2}\Delta^{2}
	\E\big[|A \hat{X}_{-k\tau+t_{j+1}}^{-k\tau}
	-
	f(t_{j+1},\hat{X}_{-k\tau+t_{j+1}}^{-k\tau})|^{2}\big]
	\\\leq&~
	\big(1-2(\lambda-\widetilde{L_{f}})(1-\theta)\Delta\big)
	\E\big[|\hat{X}_{-k\tau+t_{j}}^{-k\tau}|^{2}\big]
	\\&~
	+
	\big((p^{*}-1)\theta - (p^{*}-2)\big)
	\Delta \E\big[|g(t_{j},
	\hat{X}_{-k\tau+t_{j}}^{-k\tau})|^{2}\big]
	\\&~+
	(1-\theta)^{2}\Delta^{2}
	\E\big[|A \hat{X}_{-k\tau+t_{j}}^{-k\tau}
	-
	f(t_{j},\hat{X}_{-k\tau+t_{j}}^{-k\tau})|^{2}\big]
	+
	2\Delta\widetilde{C}
	\\\leq&~
	\big(1+2(\lambda-\widetilde{L_{f}})\theta\Delta
	-2(\lambda-\widetilde{L_{f}})\theta\Delta\big)
	\E\big[|\hat{X}_{-k\tau+t_{j}}^{-k\tau}|^{2}\big]
	\\&~
	+
	\Big((p^{*}-1)\theta - \frac{p^{*}-2}{2}\Big)
	\Delta \E\big[|g(t_{j},
	\hat{X}_{-k\tau+t_{j}}^{-k\tau})|^{2}\big]
	\\&~+
	(\theta^{2}-2\theta+1)\Delta^{2}
	\E\big[|A \hat{X}_{-k\tau+t_{j}}^{-k\tau}
	-
	f(t_{j},\hat{X}_{-k\tau+t_{j}}^{-k\tau})|^{2}\big]
	+
	2\Delta\widetilde{C}.
\end{align*}
Setting
\begin{align*}
	\widetilde{C_{\Delta}}
	:=
	\max\bigg\{1-
	\frac{2(\lambda-\widetilde{L_{f}})\theta\Delta}
	{1+2(\lambda-\widetilde{L_{f}})\theta\Delta},~
	1-\frac{p^{*}-2}{2(p^{*}-1)\theta},~
	1-\frac{2\theta-1}{\theta^{2}}\bigg\}
	\in [0,1) 
\end{align*}
leads to
\begin{align*}
	&~\big(1+2(\lambda-\widetilde{L_{f}})\theta\Delta\big)
	\E\big[|\hat{X}_{-k\tau+t_{j+1}}^{-k\tau}|^{2}\big]
	+
	(p^{*}-1)\theta\Delta\E\big[|g(t_{j+1},
	\hat{X}_{-k\tau+t_{j+1}}^{-k\tau})|^{2}\big]
	\\&~
	+
	\theta^{2}\Delta^{2}
	\E\big[|A \hat{X}_{-k\tau+t_{j+1}}^{-k\tau}
	-
	f(t_{j+1},\hat{X}_{-k\tau+t_{j+1}}^{-k\tau})|^{2}\big]
	-
	\frac{2\Delta\widetilde{C}}{1-\widetilde{C_{\Delta}}}
	\\\leq&~
	\widetilde{C_{\Delta}}
	\Big(\big(1+2(\lambda-\widetilde{L_{f}})
	\theta\Delta\big)
	\E\big[|\hat{X}_{-k\tau+t_{j}}^{-k\tau}|^{2}\big]
	+
	(p^{*}-1)\theta
	\Delta \E\big[|g(t_{j},
	\hat{X}_{-k\tau+t_{j}}^{-k\tau})|^{2}\big]
	\\&~+
	\theta^{2} \Delta^{2}
	\E\big[|A \hat{X}_{-k\tau+t_{j}}^{-k\tau}
	-
	f(t_{j},\hat{X}_{-k\tau+t_{j}}^{-k\tau})|^{2}\big]
	-
	\frac{2\Delta\widetilde{C}}{1-\widetilde{C_{\Delta}}} \Big).
\end{align*}
By iteration, we use \eqref{eq:fgsuperlinear} to obtain
\begin{align*}
	&~\big(1+2(\lambda-\widetilde{L_{f}})
	\theta\Delta\big)
	\E\big[|\hat{X}_{-k\tau+t_{j}}^{-k\tau}|^{2}\big]
	-
	\frac{2\Delta\widetilde{C}}{1-\widetilde{C_{\Delta}}}
	\\\leq&~
	\widetilde{C_{\Delta}}^{j}
	\bigg(\big(1+2(\lambda-\widetilde{L_{f}})
	\theta\Delta\big)\E\big[|\xi|^{2}\big]
	+
	(p^{*}-1)\theta
	\Delta \E\big[|g(0,\xi)|^{2}\big]
	\\&~+
	\theta^{2} \Delta^{2}
	\E\big[|A \xi
	-
	f(0,\xi)|^{2}\big]
	-
	\frac{2\Delta\widetilde{C}}
	{1-\widetilde{C_{\Delta}}} \bigg)
	\\\leq&~
	C\big(1 + \E\big[|\xi|^{2}\big]
	+
	\E\big[|\xi|^{\gamma+1}\big]
	+
	\E\big[|\xi|^{2\gamma}\big]\big).
\end{align*}
Together with
%
\begin{align*}
	\frac{\Delta}{1-\widetilde{C_{\Delta}}}
	=&~
	\Delta\max\bigg\{
	\frac{1+2(\lambda-\widetilde{L_{f}})\theta\Delta}
	{2(\lambda-\widetilde{L_{f}})\theta\Delta},~
	\frac{2(p^{*}-1)\theta}{p^{*}-2},~
	\frac{\theta^{2}}{2\theta-1}\bigg\}
	\\=&~
	\max\bigg\{
	\frac{1+2(\lambda-\widetilde{L_{f}})\theta\Delta}
	{2(\lambda-\widetilde{L_{f}})\theta},~
	\frac{2(p^{*}-1)\theta\Delta}{p^{*}-2},~
	\frac{\theta^{2}\Delta}{2\theta-1}\bigg\}
	\\\leq&~
	\max\bigg\{
	\frac{1+2(\lambda-\widetilde{L_{f}})\theta}
	{2(\lambda-\widetilde{L_{f}})\theta},~
	{\frac{2(p^{*}-1)\theta}{p^{*}-2}},~
	\frac{\theta^{2}}{2\theta-1}\bigg\}
	:= \bar{C},
\end{align*}
we derive
\begin{align*}
	\E\big[|\hat{X}_{-k\tau+t_{j}}^{-k\tau}|^{2}\big]
	\leq
	C\big(1 + \E\big[|\xi|^{2}\big]
	+
	\E\big[|\xi|^{\gamma+1}\big]
	+
	\E\big[|\xi|^{2\gamma}\big]\big)
	+
	2\bar{C}\widetilde{C},
\end{align*}
and finish the proof.
\end{proof}

The following lemma indicates that any two numerical solutions starting from different initial conditions can be arbitrarily close after sufficiently many iterations.

\begin{lemma}\label{lem:numrob}
      Let $\{\hat{X}_{-k\tau+t_{j}}^{-k\tau}\}$ and $\{\hat{Y}_{-k\tau+t_{j}}^{-k\tau}\}$ be two solutions of {the ST methods \eqref{eq:BEM}} with different initial values $\xi$ and $\eta$, respectively. Suppose that Assumption \ref{asm:L} holds for both initial values $\xi$ and $\eta$. Then {there exists a constant $C > 0$}, independent of $j$ and $k$, such that
	\begin{equation}\label{lem:numrob:rs}
		\E\big[|\hat{X}_{-k\tau+t_{j}}^{-k\tau}
        - \hat{Y}_{-k\tau+t_{j}}^{-k\tau}|^{2}\big]
		\leq
		CC_{\Delta}^{j} \Big(\E[|\xi-\eta|^{2}]
		+
        \big(\E\big[|\xi-\eta|^{2\gamma}\big]
        \big)^{\frac{1}{\gamma}} \Big),
	\end{equation}
	where
	\begin{align}\label{eq:CDelta}
		C_{\Delta}
		:=
		\max\bigg\{
		1-\frac{2(\lambda-L_{f})\theta\Delta}
		{1+2(\lambda-L_{f})\theta\Delta},~
		1-\frac{p^{*}-2}{2(p^{*}-1)\theta},~
		1-\frac{2\theta-1}{\theta^{2}}\bigg\} \in [0,1).
	\end{align}
\end{lemma}

\begin{proof}
To simplify the notation, we denote
$E_{-k\tau+t_{j}}^{-k\tau}:=\hat{X}_{-k\tau+t_{j}}^{-k\tau} -\hat{Y}_{-k\tau+t_{j}}^{-k\tau}$ and apply \eqref{eq:BEM} as well as the arguments used in \eqref{eq:hatXerror} to get
\begin{align}\label{lem:numrob:1}
	\nonumber
	&~\E\big[| E_{-k\tau+t_{j+1}}^{-k\tau}
	- \theta\Delta ( -A E_{-k\tau+t_{j+1}}^{-k\tau}
	+ f(t_{j+1}, \hat{X}_{-k\tau+t_{j+1}}^{-k\tau})
	-f(t_{j+1}, \hat{Y}_{-k\tau+t_{j+1}}^{-k\tau}))|^{2}\big]
	\nonumber
	\\=&~\nonumber
	\E\big[| E_{-k\tau+t_{j}}^{-k\tau}
	+ (1-\theta)\Delta ( -A E_{-k\tau+t_{j}}^{-k\tau}
	+ f(t_{j}, \hat{X}_{-k\tau+t_{j}}^{-k\tau})
	-f(t_{j}, \hat{Y}_{-k\tau+t_{j}}^{-k\tau}))
	\\&~+
	(g(t_{j}, \hat{X}_{-k\tau+t_{j}}^{-k\tau})
	-g(t_{j}, \hat{Y}_{-k\tau+t_{j}}^{-k\tau}))
	{\Delta W_{-k\tau+t_{j}}} |^{2}\big]
	\\=&~\nonumber
	\E\big[| E_{-k\tau+t_{j}}^{-k\tau}
	+ (1-\theta)\Delta ( -A E_{-k\tau+t_{j}}^{-k\tau}
	+ f(t_{j}, \hat{X}_{-k\tau+t_{j}}^{-k\tau})
	-f(t_{j}, \hat{Y}_{-k\tau+t_{j}}^{-k\tau}))|^{2}\big]
	\\&~+\nonumber
	\Delta\E\big[|g(t_{j}, \hat{X}_{-k\tau+t_{j}}^{-k\tau})
	-g(t_{j}, \hat{Y}_{-k\tau+t_{j}}^{-k\tau})|^{2}\big].
\end{align}
It follows from \eqref{eq:Ainequality} and {\eqref{eq:fgmonotonicity}} that
\begin{align}\label{eq:410410}
	\notag
	&~\E\big[| E_{-k\tau+t_{j+1}}^{-k\tau}
	- \theta\Delta ( -A E_{-k\tau+t_{j+1}}^{-k\tau}
	+ f(t_{j+1}, \hat{X}_{-k\tau+t_{j+1}}^{-k\tau})
	-f(t_{j+1}, \hat{Y}_{-k\tau+t_{j+1}}^{-k\tau}))|^{2}\big]
	\\=&~\notag
	\E\big[| E_{-k\tau+t_{j+1}}^{-k\tau}|^{2}\big]
	+ \theta^{2}\Delta^{2} \E\big[|
	-A E_{-k\tau+t_{j+1}}^{-k\tau}
	+ f(t_{j+1}, \hat{X}_{-k\tau+t_{j+1}}^{-k\tau})
	-f(t_{j+1}, \hat{Y}_{-k\tau+t_{j+1}}^{-k\tau})|^{2}\big]
	\\&~
	-2\theta\Delta \E\big[\big\langle
	E_{-k\tau+t_{j+1}}^{-k\tau},
	-A E_{-k\tau+t_{j+1}}^{-k\tau}
	+ f(t_{j+1}, \hat{X}_{-k\tau+t_{j+1}}^{-k\tau})
	-f(t_{j+1}, \hat{Y}_{-k\tau+t_{j+1}}^{-k\tau})
	\big\rangle\big]
	\\\geq&~\notag
	\big(1+2(\lambda-L_{f})\theta\Delta\big)
	\E\big[| E_{-k\tau+t_{j+1}}^{-k\tau}|^{2}\big]
	\\&~+ \notag
	\theta^{2}\Delta^{2} \E\big[|
	-A E_{-k\tau+t_{j+1}}^{-k\tau}
	+ f(t_{j+1}, \hat{X}_{-k\tau+t_{j+1}}^{-k\tau})
	-f(t_{j+1}, \hat{Y}_{-k\tau+t_{j+1}}^{-k\tau})|^{2}\big]
	\\&~+\notag
	2(p^{*}-1)\theta\Delta\E\big[|
	g(t_{j+1}, \hat{X}_{-k\tau+t_{j+1}}^{-k\tau})
	-g(t_{j+1}, \hat{Y}_{-k\tau+t_{j+1}}^{-k\tau})|^{2}\big].
\end{align}
Similarly, we have
\begin{align}\label{eq:411411}
	\notag
	&~\E\big[| E_{-k\tau+t_{j}}^{-k\tau}
	+ (1-\theta)\Delta ( -A E_{-k\tau+t_{j}}^{-k\tau}
	+ f(t_{j}, \hat{X}_{-k\tau+t_{j}}^{-k\tau})
	- f(t_{j}, \hat{Y}_{-k\tau+t_{j}}^{-k\tau}))|^{2}\big]
	\\=&~\notag
	\E\big[| E_{-k\tau+t_{j}}^{-k\tau}|^{2}\big]
	+
	(1-\theta)^{2}\Delta^{2}
	\E\big[|-A E_{-k\tau+t_{j}}^{-k\tau}
	+ f(t_{j}, \hat{X}_{-k\tau+t_{j}}^{-k\tau})
	- f(t_{j}, \hat{Y}_{-k\tau+t_{j}}^{-k\tau})|^{2}\big]
	\\&~+\notag
	2(1-\theta)\Delta\E\big[\big\langle
	E_{-k\tau+t_{j}}^{-k\tau},
	-A E_{-k\tau+t_{j}}^{-k\tau}
	+ f(t_{j}, \hat{X}_{-k\tau+t_{j}}^{-k\tau})
	- f(t_{j}, \hat{Y}_{-k\tau+t_{j}}^{-k\tau})
	\big\rangle\big]
	\\\leq&~
	\big(1-2(\lambda-L_{f})(1-\theta)\Delta\big)
	\E\big[| E_{-k\tau+t_{j}}^{-k\tau}|^{2}\big]
	\\&~+\notag
	(1-\theta)^{2}\Delta^{2}
	\E\big[|-A E_{-k\tau+t_{j}}^{-k\tau}
	+ f(t_{j}, \hat{X}_{-k\tau+t_{j}}^{-k\tau})
	- f(t_{j}, \hat{Y}_{-k\tau+t_{j}}^{-k\tau})|^{2}\big]
	\\&~\notag
	-
	2(p^{*}-1)(1-\theta)\Delta
	\E\big[|g(t_{j}, \hat{X}_{-k\tau+t_{j}}^{-k\tau})
	- g(t_{j}, \hat{Y}_{-k\tau+t_{j}}^{-k\tau})|^{2}\big].
\end{align}
Inserting \eqref{eq:410410} and \eqref{eq:411411} into \eqref{lem:numrob:1} results in
\begin{align*}
	&~\big(1+2(\lambda-L_{f})\theta\Delta\big)
	\E\big[| E_{-k\tau+t_{j+1}}^{-k\tau}|^{2}\big]
	\\&~+
	\theta^{2}\Delta^{2} \E\big[|
	-A E_{-k\tau+t_{j+1}}^{-k\tau}
	+ f(t_{j+1}, \hat{X}_{-k\tau+t_{j+1}}^{-k\tau})
	-f(t_{j+1}, \hat{Y}_{-k\tau+t_{j+1}}^{-k\tau})|^{2}\big]
	\\&~+2(p^{*}-1)\theta\Delta\E\big[|
	g(t_{j+1}, \hat{X}_{-k\tau+t_{j+1}}^{-k\tau})
	-g(t_{j+1}, \hat{Y}_{-k\tau+t_{j+1}}^{-k\tau})|^{2}\big]
	\\\leq&~
	\big(1-2(\lambda-L_{f})(1-\theta)\Delta\big)
	\E\big[| E_{-k\tau+t_{j}}^{-k\tau}|^{2}\big]
	\\&~+
	(1-\theta)^{2}\Delta^{2}
	\E\big[|-A E_{-k\tau+t_{j}}^{-k\tau}
	+ f(t_{j}, \hat{X}_{-k\tau+t_{j}}^{-k\tau})
	- f(t_{j}, \hat{Y}_{-k\tau+t_{j}}^{-k\tau})|^{2}\big]
	\\&~
	+\big(1-2(p^{*}-1)(1-\theta)\big)\Delta
	\E\big[|g(t_{j}, \hat{X}_{-k\tau+t_{j}}^{-k\tau})
	- g(t_{j}, \hat{Y}_{-k\tau+t_{j}}^{-k\tau})|^{2}\big]
	\\\leq&~
	\big(1+2(\lambda-L_{f})\theta\Delta
	-2(\lambda-L_{f})\theta\Delta\big)
	\E\big[| E_{-k\tau+t_{j}}^{-k\tau}|^{2}\big]
	\\&~+
	(\theta^{2}-2\theta+1) \Delta^{2}
	\E\big[|-A E_{-k\tau+t_{j}}^{-k\tau}
	+ f(t_{j}, \hat{X}_{-k\tau+t_{j}}^{-k\tau})
	- f(t_{j}, \hat{Y}_{-k\tau+t_{j}}^{-k\tau})|^{2}\big]
	\\&~
	+\big(2(p^{*}-1)\theta-(p^{*}-2)\big)\Delta
	\E\big[|g(t_{j}, \hat{X}_{-k\tau+t_{j}}^{-k\tau})
	- g(t_{j}, \hat{Y}_{-k\tau+t_{j}}^{-k\tau})|^{2}\big],
\end{align*}
which together with \eqref{eq:CDelta} implies
\begin{align*}
	&~\big(1+2(\lambda-L_{f})\theta\Delta\big)
	\E\big[| E_{-k\tau+t_{j+1}}^{-k\tau}|^{2}\big]
	\\&~+
	\theta^{2}\Delta^{2} \E\big[|
	-A E_{-k\tau+t_{j+1}}^{-k\tau}
	+ f(t_{j+1}, \hat{X}_{-k\tau+t_{j+1}}^{-k\tau})
	-f(t_{j+1}, \hat{Y}_{-k\tau+t_{j+1}}^{-k\tau})|^{2}\big]
	\\&~+2(p^{*}-1)\theta\Delta\E\big[|
	g(t_{j+1}, \hat{X}_{-k\tau+t_{j+1}}^{-k\tau})
	-g(t_{j+1}, \hat{Y}_{-k\tau+t_{j+1}}^{-k\tau})|^{2}\big]
	\\\leq&~
	C_{\Delta} \Big(
	\big(1+2(\lambda-L_{f})\theta\Delta\big)
	\E\big[| E_{-k\tau+t_{j}}^{-k\tau}|^{2}\big]
	\\&~+
	\theta^{2} \Delta^{2}
	\E\big[|-A E_{-k\tau+t_{j}}^{-k\tau}
	+ f(t_{j}, \hat{X}_{-k\tau+t_{j}}^{-k\tau})
	- f(t_{j}, \hat{Y}_{-k\tau+t_{j}}^{-k\tau})|^{2}\big]
	\\&~
	+2(p^{*}-1)\theta\Delta
	\E\big[|g(t_{j}, \hat{X}_{-k\tau+t_{j}}^{-k\tau})
	- g(t_{j}, \hat{Y}_{-k\tau+t_{j}}^{-k\tau})|^{2}\big]\Big).
\end{align*}

By \eqref{eq:fpolynomial}, \eqref{eq:gpolynomial} and the H\"{o}lder inequality, one gets
\begin{align}
	&~\E\big[| E_{-k\tau+t_{j}}^{-k\tau}|^{2}\big] \notag
	\\\leq&~\notag
	C_{\Delta}^{j} \Big(
	\big(1+2(\lambda-L_{f})\theta\Delta\big)
	\E\big[|\xi-\eta|^{2}\big]
	+
	2(p^{*}-1)\theta\Delta
	\E\big[|g(0, \xi) - g(0, \eta)|^{2}
	\\&~+\notag
	\theta^{2} \Delta^{2}
	\E\big[|-A (\xi-\eta) + f(0, \xi)
	- f(0, \eta)|^{2}\big]\Big)
	\\\leq&~\notag
	CC_{\Delta}^{j}\big(\E\big[|\xi-\eta|^{2}\big]
	+
	\E\big[|f(0, \xi) - f(0, \eta)|^{2}\big]
	+
	\E\big[|g(0, \xi) - g(0, \eta)|^{2}\big]\big)
	\\\leq&~\notag
	CC_{\Delta}^{j}\big(\E\big[|\xi-\eta|^{2}\big]
	+
	\E\big[(1+|\xi|+|\eta|)^{2\gamma-2} |\xi-\eta|^{2}\big]
	+
	\E\big[(1+|\xi|+|\eta|)^{\gamma-1} |\xi-\eta|^{2}\big]\big)
	\\\leq&~\notag
	CC_{\Delta}^{j}
	\Big(\E\big[|\xi-\eta|^{2}\big]
	+
	\big(\E\big[(1+|\xi|+|\eta|)^{2\gamma}\big]
	\big)^{\frac{\gamma-1}{\gamma}}
	\big(\E\big[|\xi-\eta|^{2\gamma}\big]\big)^{\frac{1}{\gamma}}
	\\&~+\notag
	\big(\E\big[(1+|\xi|+|\eta|)^{\gamma}\big]
	\big)^{\frac{\gamma-1}{\gamma}}
	\big(\E\big[|\xi-\eta|^{2\gamma}\big]
	\big)^{\frac{1}{\gamma}}\Big),
	\\\leq&~\notag
	CC_{\Delta}^{j}
	\Big(\E\big[|\xi-\eta|^{2}\big]
	+
	\big(\E\big[|\xi-\eta|^{2\gamma}\big]
	\big)^{\frac{1}{\gamma}}\Big),
\end{align}
which indicates the required result and therefore completes the proof.
\end{proof}

Following the arguments used in \cite[Theorem 3.4]{feng2017numerical}, \cite[Theorem 2.1]{rong2020numerical} and \cite[Theorem 8]{wu2023backward}, we obtain the existence and uniqueness of random periodic solutions for the ST methods.

\begin{theorem}\label{thm:nump}
	Suppose that Assumption \ref{asm:L} holds. Then for any initial value $\xi$ satisfying \rm{(A4)}, the considered method \eqref{eq:BEM} admits a unique random period solution $\hat{X}_{t}^{*} \in L^{2}(\Omega;\R^{d}),t \geq 0$ such that
	\begin{equation}\label{thm:ep:rsrere}
		\lim_{k \to \infty}
		|\hat{X}_{t}^{-k\tau}
		- \hat{X}_{t}^{*}|_{L^{2}(\Omega;\R^{d})}
		=
		0.
	\end{equation}
\end{theorem}

\section{Mean square convergence order of ST methods}
\label{eq:convergence}
In this section, we will establish the mean square convergence order of ST methods. To this end, let us define
\begin{align}\label{eq:errorRkj}
	\mathcal{R}_{k,j}
	:= &~\notag
	\theta \int_{-k\tau+t_{j}}^{-k\tau+t_{j+1}}
	(-A) \big( X_{s}^{-k\tau}
	-X_{-k\tau+t_{j+1}}^{-k\tau} \big)
	+
	\big(f(s,X_{s}^{-k\tau})
	-f(t_{j+1},X_{-k\tau+t_{j+1}}^{-k\tau})\big)\diff{s}
	\\&~
	+(1-\theta) \int_{-k\tau+t_{j}}^{-k\tau+t_{j+1}}
	(-A)\big(X_{s}^{-k\tau}-X_{-k\tau+t_{j}}^{-k\tau}\big)
	+
	\big(f(s,X_{s}^{-k\tau})
	-f(t_{j},X_{-k\tau+t_{j}}^{-k\tau})\big)\diff{s}
	\\&~\notag
	+\int_{-k\tau+t_{j}}^{-k\tau+t_{j+1}}
	g(s,X_{s}^{-k\tau})-
	g(t_{j},X_{-k\tau+t_{j}}^{-k\tau})\diff{W_{s}}
\end{align}
for all $k, j \in \N$. The following result provides uniform bounded estimates for the second moment of $\mathcal{R}_{k,j}$ and its conditional expectation $\E(\mathcal{R}_{k,j} ~|~\F_{-k\tau+t_{j}})$.

\begin{lemma}\label{eq:RjRjjFjestimate}
	Suppose that Assumption \ref{asm:L} holds. Then for any $k, j \in \N$, there exists a constant $C > 0$, independent of $k,j$, such that
	\begin{align}\label{eq:RjrJF}
		\E\big[|\mathcal{R}_{k,j}|^{2}\big]
		\leq
		C\Delta^{2},
		\quad
		\E\big[|\E(\mathcal{R}_{k,j} ~|~\F_{-k\tau+t_{j}})|^{2}\big]
		\leq
		C\Delta^{3}.
	\end{align}
\end{lemma}

\begin{proof}
By the H\"{o}lder inequality and the It\^{o} isometry, we obtain
\begin{align}\label{eq:Rjestimate}
	\E\big[|\mathcal{R}_{k,j}|^{2}\big]
	\leq&~\notag
	3\theta^{2} \E\bigg[\Big|\int_{-k\tau+t_{j}}^{-k\tau+t_{j+1}}
	(-A) \big( X_{s}^{-k\tau}
	-X_{-k\tau+t_{j+1}}^{-k\tau} \big)
	\\&~+\notag
	\big(f(s,X_{s}^{-k\tau}) - f(t_{j+1},
	X_{-k\tau+t_{j+1}}^{-k\tau})\big)\diff{s}\Big|^{2}\bigg]
	\\&~\notag
	+3(1-\theta)^{2}
	\E\bigg[\Big|\int_{-k\tau+t_{j}}^{-k\tau+t_{j+1}}
	(-A)\big(X_{s}^{-k\tau}-X_{-k\tau+t_{j}}^{-k\tau}\big)
	\\&~+\notag
	\big(f(s,X_{s}^{-k\tau})
	-f(t_{j},X_{-k\tau+t_{j}}^{-k\tau})\big)\diff{s}\Big|^{2}\bigg]
	\\&~
	+3\E\bigg[\Big|\int_{-k\tau+t_{j}}^{-k\tau+t_{j+1}}
	g(s,X_{s}^{-k\tau})-
	g(t_{j},X_{-k\tau+t_{j}}^{-k\tau})\diff{W_{s}}\Big|^{2}\bigg]
	\\\leq&~\notag
	C\Delta\int_{-k\tau+t_{j}}^{-k\tau+t_{j+1}}
	\E\big[|(-A)(X_{s}^{-k\tau}
	-X_{-k\tau+t_{j+1}}^{-k\tau})|^{2}\big]
	\\&~+\notag
	\E\big[|f(s,X_{s}^{-k\tau}) - f(t_{j+1},
	X_{-k\tau+t_{j+1}}^{-k\tau})|^{2}\big]\diff{s}
	\\&~\notag
	+C\Delta\int_{-k\tau+t_{j}}^{-k\tau+t_{j+1}}
	\E\big[|(-A)(X_{s}^{-k\tau}
	-X_{-k\tau+t_{j}}^{-k\tau})|^{2}\big]
	\\&~+\notag
	\E\big[|f(s,X_{s}^{-k\tau})
	-f(t_{j},X_{-k\tau+t_{j}}^{-k\tau})|^{2}\big]\diff{s}
	\\&~\notag
	+C\int_{-k\tau+t_{j}}^{-k\tau+t_{j+1}}
	\E\big[|g(s,X_{s}^{-k\tau})-
	g(t_{j},X_{-k\tau+t_{j}}^{-k\tau})|^{2}\big]\diff{s}.
\end{align}
Now for any $s \in [-k\tau+t_{j},-k\tau+t_{j+1}]$, using \eqref{eq:fpolynomial}, \eqref{eq:gpolynomial},
{Lemmas \ref{lem:exactbound}  and \ref{lem:exactholder}} indicates
\begin{align}\label{eq:fholder}
	&~\notag
	\E\big[|f(s,X_{s}^{-k\tau}) - f(t_{j+1},
	X_{-k\tau+t_{j+1}}^{-k\tau})|^{2}\big]
	\\=&~\notag
	\E\big[|f(s,X_{s}^{-k\tau}) - f(-k\tau+t_{j+1},
	X_{-k\tau+t_{j+1}}^{-k\tau})|^{2}\big]
	\\\leq&~\notag
	C\E\big[\big(1 + |X_{s}^{-k\tau}|
	+ |X_{-k\tau+t_{j+1}}^{-k\tau}|\big)^{2(\gamma-1)}
	|X_{s}^{-k\tau}-X_{-k\tau+t_{j+1}}^{-k\tau}|^{2}\big]
	\\&~+
	C\E\big[\big(1 + |X_{s}^{-k\tau}|
	+ |X_{-k\tau+t_{j+1}}^{-k\tau}|\big)^{2\gamma}
	|s-(-k\tau+t_{j+1})|^{2}\big]
	\\\leq&~\notag
	C\big(\E\big[\big(1 + |X_{s}^{-k\tau}|
	+ |X_{-k\tau+t_{j+1}}^{-k\tau}|\big)^{4\gamma-2}\big]
	\big)^{\frac{2(\gamma-1)}{4\gamma-2}}
	\\&~\times\notag
	\big(\E\big[|X_{s}^{-k\tau}
	-X_{-k\tau+t_{j+1}}^{-k\tau}|^{\frac{4\gamma-2}{\gamma}}\big]
	\big)^{\frac{2\gamma}{4\gamma-2}}
	\\&~+\notag
	C\Delta^{2}\E\big[\big(1 + |X_{s}^{-k\tau}|
	+ |X_{-k\tau+t_{j+1}}^{-k\tau}|\big)^{2\gamma}\big]
	\\\leq& \notag
	C\Delta,
\end{align}
and similarly
\begin{gather}
	\label{eq:ffholder}
	\E\big[|f(s,X_{s}^{-k\tau})
	-f(t_{j},X_{-k\tau+t_{j}}^{-k\tau})|^{2}\big]
	\leq C\Delta,
	\\
	\label{eq:gholder}
	\E\big[|g(s,X_{s}^{-k\tau})-
	g(t_{j},X_{-k\tau+t_{j}}^{-k\tau})|^{2}\big]
	\leq
	C\Delta.
\end{gather}
Plugging \eqref{eq:fholder}, \eqref{eq:ffholder} and \eqref{eq:gholder} into \eqref{eq:Rjestimate} promises the first part of \eqref{eq:RjrJF}. For the second part of \eqref{eq:RjrJF}, we use the Jensen inequality for conditional expectation to get
\begin{align*}
	&~\E\big[|\E(\mathcal{R}_{k,j} ~|~\F_{-k\tau+t_{j}})|^{2}\big]
	\\\leq&~
	2\theta^{2}\E\bigg[\Big|
	\E\Big(\int_{-k\tau+t_{j}}^{-k\tau+t_{j+1}}
	(-A) \big( X_{s}^{-k\tau}
	-X_{-k\tau+t_{j+1}}^{-k\tau} \big)
	\\&~+
	\big(f(s,X_{s}^{-k\tau})
	-f(t_{j+1},X_{-k\tau+t_{j+1}}^{-k\tau})\big)\diff{s}
	~|~\F_{-k\tau+t_{j}}\Big)
	\Big|^{2}\bigg]
	\\&~+
	2(1-\theta)^{2}\E\bigg[\Big|
	\E\Big(\int_{-k\tau+t_{j}}^{-k\tau+t_{j+1}}
	(-A)\big(X_{s}^{-k\tau}-X_{-k\tau+t_{j}}^{-k\tau}\big)
	\\&~+
	\big(f(s,X_{s}^{-k\tau})
	-f(t_{j},X_{-k\tau+t_{j}}^{-k\tau})\big)\diff{s}
	~|~\F_{-k\tau+t_{j}}\Big)
	\Big|^{2}\bigg]
	\\\leq&~
	2\E\bigg[\E\Big(\Big|\int_{-k\tau+t_{j}}^{-k\tau+t_{j+1}}
	(-A) \big( X_{s}^{-k\tau}
	-X_{-k\tau+t_{j+1}}^{-k\tau} \big)
	\\&~+
	\big(f(s,X_{s}^{-k\tau})
	-f(t_{j+1},X_{-k\tau+t_{j+1}}^{-k\tau})
	\big)\diff{s} \Big|^{2}
	~|~\F_{-k\tau+t_{j}}\Big)
	\bigg]
	\\&~+
	2\E\bigg[\E\Big(\Big|\int_{-k\tau+t_{j}}^{-k\tau+t_{j+1}}
	(-A)\big(X_{s}^{-k\tau}-X_{-k\tau+t_{j}}^{-k\tau}\big)
	\\&~+
	\big(f(s,X_{s}^{-k\tau})
	-f(t_{j},X_{-k\tau+t_{j}}^{-k\tau})\big)\diff{s}\Big|^{2}
	~|~\F_{-k\tau+t_{j}}\Big) \bigg]
	\\=&~
	2\E\bigg[\Big|\int_{-k\tau+t_{j}}^{-k\tau+t_{j+1}}
	(-A) \big( X_{s}^{-k\tau}
	-X_{-k\tau+t_{j+1}}^{-k\tau} \big)
	\\&~+
	\big(f(s,X_{s}^{-k\tau})
	-f(t_{j+1},X_{-k\tau+t_{j+1}}^{-k\tau})
	\big)\diff{s}\Big|^{2}\bigg]
	\\&~+
	2\E\bigg[\Big|\int_{-k\tau+t_{j}}^{-k\tau+t_{j+1}}
	(-A)\big(X_{s}^{-k\tau}-X_{-k\tau+t_{j}}^{-k\tau}\big)
	\\&~+
	\big(f(s,X_{s}^{-k\tau})
	-f(t_{j},X_{-k\tau+t_{j}}^{-k\tau})\big)\diff{s}\Big|^{2}\bigg].
\end{align*}
Repeating the techniques used in \eqref{eq:Rjestimate} and exploiting
\eqref{eq:fholder}, \eqref{eq:ffholder} as well as \eqref{eq:gholder} {imply} the second part of \eqref{eq:RjrJF}. Thus we complete the proof.
\end{proof}

It is expected that the orders of estimates with respect to the stepsize $\Delta$ in \eqref{eq:RjrJF} should be higher when SDEs \eqref{eq:SDE} driven by additive noise. We will show that this happens at the price of requiring a differentiable condition on the drift coefficient. Following the idea presented in \cite{guo2023order}, we make the following assumption.

\begin{assumption}\label{asm:Ladditive}
	Suppose that the diffusion coefficient function $g \colon \R \to \R^{d \times m}$ is continuous and periodic in time with period $\tau > 0$, i.e., $g(t+\tau) = g(t)$ for all $t \in \R$. Besides, there exists a constant $C > 0$ such that $\sup\limits_{t \in [0,\tau)}|g(t)| \leq C$ and
	\begin{equation}
		|g(t)-g(s)| \leq C|t-s|, \quad s,t \in [0,\tau).
	\end{equation}
	Moreover, assume that the drift coefficient function $f \colon \R \times \R^{d} \to \R^{d}$ is continuously differentiable, and that there exists a constant $C > 0$ such that
	\begin{align}
		\bigg|\bigg(\frac{\partial{f}(t,y)}{\partial{x}}
		-
		\frac{\partial{f}(t,z)}{\partial{x}}\bigg)u\bigg|
		\leq
		C(1+|y|+|z|)^{\gamma-2}|y-z||u|,
		\quad
		y,z,u \in \R^{d},
	\end{align}
	where $\gamma  \geq 1$ comes from \eqref{eq:fpolynomial},
    and $\frac{\partial f}{\partial x}$ denotes the partial derivative of $f$ with respect to the state variable $x$.
\end{assumption}

Under the above additional assumption, one can repeat the arguments used in \cite[Theorem 4.6]{guo2023order} to improve the estimates in Lemma \ref{eq:RjRjjFjestimate}. The proof of the following lemma is thus omitted.

\begin{lemma}\label{eq:RjRjjFjestimateadditive}
	Suppose that Assumptions \ref{asm:L} and \ref{asm:Ladditive} hold. Then for any $k, j \in \N$, there exists a constant $C > 0$, independent of $k,j$, such that
	\begin{align}\label{eq:RjrJFadditive}
		\E\big[|\mathcal{R}_{k,j}|^{2}\big]
		\leq
		C\Delta^{3},
		\quad
		\E\big[|\E(\mathcal{R}_{k,j} ~|~\F_{-k\tau+t_{j}})|^{2}\big]
		\leq
		C\Delta^{4}.
	\end{align}
\end{lemma}

Now we are in a position to present a uniform estimate for the error between the numerical solution $\hat{X}_{-k\tau+t_{j}}^{-k\tau}$ and the exact solution $X_{-k\tau+t_{j}}^{-k\tau}$ in the mean square sense.

\begin{lemma}\label{thm:error}
	Suppose that Assumption \ref{asm:L} holds. Then there exists a constant $C > 0$, independent of $j,k$ and $\Delta$, such that
	\begin{equation}\label{eq:daoshuresult}
		\sup_{k,j \in \mathbb{N}}
		|X_{-k\tau+t_{j}}^{-k\tau}
		-\hat{X}_{-k\tau+t_{j}}^{-k\tau}|_{L^{2}(\Omega;\R^{d})}
		\leq
		C\Delta^{\frac{1}{2}}.
	\end{equation}
	If in addition Assumption \ref{asm:Ladditive} holds, then there exists a constant $C > 0$, independent of $j,k$ and $\Delta$, such that
	\begin{equation}\label{eq:daoshuresultadditive}
		\sup_{k,j \in \mathbb{N}}
		|X_{-k\tau+t_{j}}^{-k\tau}
		-\hat{X}_{-k\tau+t_{j}}^{-k\tau}|_{L^{2}(\Omega;\R^{d})}
		\leq
		C\Delta.
	\end{equation}
\end{lemma}

\begin{proof}
For any $k,j \in \N$ and $t \in \R$, let us first denote
\begin{gather*}
	f_{-k\tau+t}^{-k\tau} := f(t,X_{-k\tau+t}^{-k\tau}),
	\qquad
	\hat{f}_{-k\tau+t_{j}}^{-k\tau}
	:=
	f(t_{j},\hat{X}_{-k\tau+t_{j}}^{-k\tau}),
	\\
	g_{-k\tau+t}^{-k\tau} := g(t,X_{-k\tau+t}^{-k\tau}),
	\qquad
	\hat{g}_{-k\tau+t_{j}}^{-k\tau}
	:=
	g(t_{j},\hat{X}_{-k\tau+t_{j}}^{-k\tau})
\end{gather*}
and $e_{-k\tau+t_{j}}^{-k\tau} := X_{-k\tau+t_{j}}^{-k\tau} - \hat{X}_{-k\tau+t_{j}}^{-k\tau}$.
It follows from \eqref{eq:SDE}, \eqref{eq:BEM} and \eqref{eq:errorRkj} that
\begin{align*}
	&~
	e_{-k\tau + t_{j+1}}^{-k\tau}
	-
	\theta\Delta \big( (-A) e_{-k\tau + t_{j+1}}^{-k\tau}
	+ f_{-k\tau + t_{j+1}}^{-k\tau}
	- \hat{f}_{-k\tau + t_{j+1}}^{-k\tau} \big)
	\nonumber
	\\ = &~
	e_{-k\tau+t_{j}}^{-k\tau}
	+ (1-\theta)\Delta \big( (-A) e_{-k\tau + t_{j}}^{-k\tau}
	+f_{-k\tau + t_{j}}^{-k\tau}
	- \hat{f}_{-k\tau + t_{j}}^{-k\tau} \big)
	\\&~+
	\big( g_{-k\tau + t_{j}}^{-k\tau}
	- \hat{g}_{-k\tau + t_{j}}^{-k\tau} \big)
    {\Delta W_{-k\tau+t_{j}}}
	+ \mathcal{R}_{k,j}.
\end{align*}
Due to that the terms $e_{-k\tau + t_{j}}^{-k\tau}$,
$f_{-k\tau + t_{j}}^{-k\tau}$,
$\hat{f}_{-k\tau + t_{j}}^{-k\tau}$,
$g_{-k\tau + t_{j}}^{-k\tau}$ and
$\hat{g}_{-k\tau + t_{j}}^{-k\tau}$ are $\F_{-k\tau+t_{j}}$-measurable,
we apply the independence between $g_{-k\tau + t_{j}}^{-k\tau} - \hat{g}_{-k\tau + t_{j}}^{-k\tau}$ and $\F_{-k\tau+t_{j}}$ to show
\begin{align}\label{lem:error:e2}
	\notag
	&~\E\big[\big|e_{-k\tau + t_{j+1}}^{-k\tau}
	- \theta\Delta \big( (-A) e_{-k\tau + t_{j+1}}^{-k\tau}
	+ f_{-k\tau + t_{j+1}}^{-k\tau}
	- \hat{f}_{-k\tau + t_{j+1}}^{-k\tau} \big)\big|^{2}\big]
	\\=&~\notag
	\E\big[\big|e_{-k\tau+t_{j}}^{-k\tau}
	+ (1-\theta)\Delta
	\big( (-A) e_{-k\tau + t_{j}}^{-k\tau}
	+ f_{-k\tau + t_{j}}^{-k\tau}
	- \hat{f}_{-k\tau + t_{j}}^{-k\tau} \big)\big|^{2}\big]
	+
	\E\big[|\mathcal{R}_{k,j}|^{2}\big]
	\\&~+
	\Delta\E\big[|g_{-k\tau + t_{j}}^{-k\tau}
	- \hat{g}_{-k\tau + t_{j}}^{-k\tau}|^{2}\big]
	+
	2\E\big[\big\langle
	\big(g_{-k\tau + t_{j}}^{-k\tau}
	- \hat{g}_{-k\tau + t_{j}}^{-k\tau} \big)
	{\Delta W_{-k\tau+t_{j}}},
	\mathcal{R}_{k,j}\big\rangle\big]
	\\&~+\notag
	2\E\big[\big\langle
	e_{-k\tau+t_{j}}^{-k\tau}
	+ (1-\theta)\Delta
	\big( (-A) e_{-k\tau + t_{j}}^{-k\tau}
	+ f_{-k\tau + t_{j}}^{-k\tau}
	- \hat{f}_{-k\tau + t_{j}}^{-k\tau} \big),
	\mathcal{R}_{k,j} \big\rangle\big].
\end{align}
As a consequence of \eqref{asm:fg-fg} and \eqref{eq:Ainequality}, we obtain
\begin{align}\label{lem:error:e2b-left}
	&~\notag
	\E\big[\big|e_{-k\tau + t_{j+1}}^{-k\tau}
	- \theta\Delta \big( (-A) e_{-k\tau + t_{j+1}}^{-k\tau}
	+ f_{-k\tau + t_{j+1}}^{-k\tau}
	- \hat{f}_{-k\tau + t_{j+1}}^{-k\tau} \big)\big|^{2}\big]
	\\=&~\notag
	\E\big[|e_{-k\tau + t_{j+1}}^{-k\tau}|^{2}\big]
	+
	\theta^{2}\Delta^{2}
	\E\big[|(-A) e_{-k\tau + t_{j+1}}^{-k\tau}
	+ f_{-k\tau + t_{j+1}}^{-k\tau}
	- \hat{f}_{-k\tau + t_{j+1}}^{-k\tau}|^{2}\big]
	\\&~
	+2\theta\Delta
	\E\big[\langle e_{-k\tau + t_{j+1}}^{-k\tau},
	A e_{-k\tau + t_{j+1}}^{-k\tau}\rangle\big]
	-2\theta\Delta
	\E\big[\langle e_{-k\tau + t_{j+1}}^{-k\tau},
	f_{-k\tau + t_{j+1}}^{-k\tau}
	- \hat{f}_{-k\tau + t_{j+1}}^{-k\tau}\rangle\big]
	\\\geq&~\notag
	\big(1 + 2(\lambda-L_{f})\theta\Delta \big)
	\E\big[|e_{-k\tau + t_{j+1}}^{-k\tau}|^{2}\big]
	+2(p^{*}-1)\theta\Delta
	\E\big[|g_{-k\tau + t_{j+1}}^{-k\tau}
	- \hat{g}_{-k\tau + t_{j+1}}^{-k\tau}|^{2}\big]
	\\&~+\notag
	\theta^{2}\Delta^{2}
	\E\big[|(-A) e_{-k\tau + t_{j+1}}^{-k\tau}
	+ f_{-k\tau + t_{j+1}}^{-k\tau}
	- \hat{f}_{-k\tau + t_{j+1}}^{-k\tau}|^{2}\big],
\end{align}
and
\begin{align}\label{lem:error:e2b-right-f}
	\notag
	&~\E\big[\big|e_{-k\tau+t_{j}}^{-k\tau}
	+ (1-\theta)\Delta
	\big( (-A) e_{-k\tau + t_{j}}^{-k\tau}
	+ f_{-k\tau + t_{j}}^{-k\tau}
	- \hat{f}_{-k\tau + t_{j}}^{-k\tau} \big)\big|^{2}\big]
	\\=&~\notag
	\E\big[|e_{-k\tau+t_{j}}^{-k\tau}|^{2}\big]
	+
	(1-\theta)^{2}\Delta^{2}
	\E\big[ |(-A) e_{-k\tau + t_{j}}^{-k\tau}
	+ f_{-k\tau + t_{j}}^{-k\tau}
	- \hat{f}_{-k\tau + t_{j}}^{-k\tau}|^{2}\big]
	\\&~
	+2(1-\theta)\Delta\E\big[\big\langle
	e_{-k\tau+t_{j}}^{-k\tau},
	(-A) e_{-k\tau + t_{j}}^{-k\tau}
	+ f_{-k\tau + t_{j}}^{-k\tau}
	- \hat{f}_{-k\tau + t_{j}}^{-k\tau}
	\big\rangle\big]
	\\\leq&~\notag
	\big(1-2(\lambda-L_{f})(1-\theta)\Delta\big)
	\E\big[|e_{-k\tau+t_{j}}^{-k\tau}|^{2}\big]
	\\&~+\notag
	(1-\theta)^{2}\Delta^{2}
	\E\big[|(-A) e_{-k\tau + t_{j}}^{-k\tau}
	+ f_{-k\tau + t_{j}}^{-k\tau}
	- \hat{f}_{-k\tau + t_{j}}^{-k\tau}|^{2}\big]
	\\&~-\notag
	2(p^{*}-1)(1-\theta)\Delta
	\E\big[|g_{-k\tau + t_{j}}^{-k\tau}
	- \hat{g}_{-k\tau + t_{j}}^{-k\tau}|^{2}\big].
\end{align}
Utilizing the weighted Young inequality $2ab \leq \kappa a^{2} + \frac{b^{2}}{\kappa}, a,b \in \R$ with $\kappa := \min\{\lambda-L_{f},2\theta-1\} > 0$ leads to
\begin{align}\label{lem:error:e2-right-g}
	\notag
	&~2\E\big[\big\langle
	\big(g_{-k\tau + t_{j}}^{-k\tau}
	- \hat{g}_{-k\tau + t_{j}}^{-k\tau} \big)
	{\Delta W_{-k\tau+t_{j}}},
	\mathcal{R}_{k,j}\big\rangle\big]
	\\\leq&~
	(p^{*}-2) \Delta
	\E\big[|g_{-k\tau + t_{j}}^{-k\tau}
	- \hat{g}_{-k\tau + t_{j}}^{-k\tau}|^{2}\big]
	+
	\frac{1}{p^{*}-2}
	\E\big[|\mathcal{R}_{k,j}|^{2}\big],
\end{align}
and
\begin{align}\label{lem:error:e2-right-f}
	&~\notag
	2\E\big[\big\langle
	e_{-k\tau+t_{j}}^{-k\tau}
	+ (1-\theta)\Delta
	\big( (-A) e_{-k\tau + t_{j}}^{-k\tau}
	+ f_{-k\tau + t_{j}}^{-k\tau}
	- \hat{f}_{-k\tau + t_{j}}^{-k\tau} \big),
	\mathcal{R}_{k,j} \big\rangle\big]
	\\=&~
	2\E\big[\big\langle
	e_{-k\tau+t_{j}}^{-k\tau}
	+ (1-\theta)\Delta
	\big( (-A) e_{-k\tau + t_{j}}^{-k\tau}
	+ f_{-k\tau + t_{j}}^{-k\tau}
	- \hat{f}_{-k\tau + t_{j}}^{-k\tau} \big),
	\E(\mathcal{R}_{k,j} ~|~\F_{-k\tau+t_{j}}) \big\rangle\big]
	\\\leq&~\notag
	\kappa\Delta
	\E\big[|e_{-k\tau+t_{j}}^{-k\tau}
	+ (1-\theta)\Delta
	\big( (-A) e_{-k\tau + t_{j}}^{-k\tau}
	+ f_{-k\tau + t_{j}}^{-k\tau}
	- \hat{f}_{-k\tau + t_{j}}^{-k\tau} \big)|^{2}\big]
	\\&~+\notag
	\frac{1}{\kappa\Delta}
	\E\big[|\E(\mathcal{R}_{k,j} ~|~\F_{-k\tau+t_{j}})|^{2}\big].
\end{align}
Substituting  \eqref{lem:error:e2b-left}, \eqref{lem:error:e2b-right-f}, \eqref{lem:error:e2-right-g} and \eqref{lem:error:e2-right-f} into \eqref{lem:error:e2} yields
%
\begin{align}\label{lem:error:lr}
	\notag
	&~\big(1 + 2(\lambda-L_{f})\theta\Delta \big)
	\E\big[|e_{-k\tau + t_{j+1}}^{-k\tau}|^{2}\big]
	\\&~\notag
	+2(p^{*}-1)\theta\Delta
	\E\big[|g_{-k\tau + t_{j+1}}^{-k\tau}
	- \hat{g}_{-k\tau + t_{j+1}}^{-k\tau}|^{2}\big]
	\\&~+\notag
	\theta^{2}\Delta^{2}
	\E\big[|(-A) e_{-k\tau + t_{j+1}}^{-k\tau}
	+ f_{-k\tau + t_{j+1}}^{-k\tau}
	- \hat{f}_{-k\tau + t_{j+1}}^{-k\tau}|^{2}\big]
	\\\leq&~
	(1+\kappa\Delta)
	\big(1-2(\lambda-L_{f})(1-\theta)\Delta\big)
	\E\big[|e_{-k\tau+t_{j}}^{-k\tau}|^{2}\big]
	\\&~+\notag
	\big(1-2(1+\kappa\Delta)(1-\theta)\big)(p^{*}-1)\Delta
	\E\big[|g_{-k\tau + t_{j}}^{-k\tau}
	- \hat{g}_{-k\tau + t_{j}}^{-k\tau}|^{2}\big]
	\\&~+\notag
	(1+\kappa\Delta)(1-\theta)^{2}\Delta^{2}
	\E\big[|(-A) e_{-k\tau + t_{j}}^{-k\tau}
	+ f_{-k\tau + t_{j}}^{-k\tau}
	- \hat{f}_{-k\tau + t_{j}}^{-k\tau}|^{2}\big]
	\\&~+\notag
	\frac{p^{*}-1}{p^{*}-2}
	\E\big[|\mathcal{R}_{k,j}|^{2}\big]
	+
	\frac{1}{\kappa\Delta}
	\E\big[|\E(\mathcal{R}_{k,j} ~|~\F_{-k\tau+t_{j}})|^{2}\big].
\end{align}
By means of Lemma \ref{eq:RjRjjFjestimate}, we obtain
\begin{align}\label{lem:error:2}
	\notag
	&~\big(1 + 2(\lambda-L_{f})\theta\Delta \big)
	\E\big[|e_{-k\tau + t_{j+1}}^{-k\tau}|^{2}\big]
	\\&~+\notag
	2(p^{*}-1)\theta\Delta
	\E\big[|g_{-k\tau + t_{j+1}}^{-k\tau}
	- \hat{g}_{-k\tau + t_{j+1}}^{-k\tau}|^{2}\big]
	\\&~+\notag
	\theta^{2}\Delta^{2}
	\E\big[|(-A) e_{-k\tau + t_{j+1}}^{-k\tau}
	+ f_{-k\tau + t_{j+1}}^{-k\tau}
	- \hat{f}_{-k\tau + t_{j+1}}^{-k\tau}|^{2}\big]
	\\\leq&~
	(1+\kappa\Delta)
	\big(1-2(\lambda-L_{f})(1-\theta)\Delta\big)
	\E\big[|e_{-k\tau+t_{j}}^{-k\tau}|^{2}\big]
	\\&~+\notag
	\big(1-2(1+\kappa\Delta)(1-\theta)\big)(p^{*}-1)\Delta
	\E\big[|g_{-k\tau + t_{j}}^{-k\tau}
	- \hat{g}_{-k\tau + t_{j}}^{-k\tau}|^{2}\big]
	\\&~+\notag
	(1+\kappa\Delta)(1-\theta)^{2}\Delta^{2}
	\E\big[|(-A) e_{-k\tau + t_{j}}^{-k\tau}
	+ f_{-k\tau + t_{j}}^{-k\tau}
	- \hat{f}_{-k\tau + t_{j}}^{-k\tau}|^{2}\big]
	+
	C\Delta^{2}.
\end{align}
From $\kappa = \min\{\lambda-L_{f},2\theta-1\} > 0$ and $2(1-\theta)^{2}\Delta \leq 1$, one gets
\begin{align*}
	&~(1+\kappa\Delta)
	\big(1-2(\lambda-L_{f})(1-\theta)\Delta\big)
	\\=&~
	1 + \kappa\Delta - 2(\lambda-L_{f})\Delta
	+ 2(\lambda-L_{f})\theta\Delta
	- 2\kappa(\lambda-L_{f})(1-\theta)\Delta^{2}
	\\\leq&~
	1 + 2(\lambda-L_{f})\theta\Delta
	- (\lambda-L_{f})\Delta
	\\=&~
	\big(1 + 2(\lambda-L_{f})\theta\Delta\big)
	\bigg(1- \frac{(\lambda-L_{f})\Delta}
	{1 + 2(\lambda-L_{f})\theta\Delta}\bigg),
\end{align*}
as well as
\begin{align*}
	1-2(1+\kappa\Delta)(1-\theta)
	=
	-1+2\theta-2\kappa\Delta(1-\theta)
	\leq
	2\theta-1
	=
	2\theta\bigg(1-\frac{1}{2\theta}\bigg),
\end{align*}
and
\begin{align*}
	(1+\kappa\Delta)(1-\theta)^{2}
	=&~
	\theta^{2} - (2\theta-1)
	+
	\kappa\Delta(1-\theta)^{2}
	\\\leq&~
	\theta^{2} - (2\theta-1)\big(1-\Delta(1-\theta)^{2}\big)
	\\\leq&~
	\theta^{2} - \frac{2\theta-1}{2}
	=
	\theta^{2}\bigg(1 - \frac{2\theta-1}{2\theta^{2}}\bigg).
\end{align*}
Setting
\begin{align}
	\overline{C_{\Delta}}
	:=
	\max\bigg\{
	1- \frac{(\lambda-L_{f})\Delta}
	{1 + 2(\lambda-L_{f})\theta\Delta},~
	1 - \frac{1}{2\theta},~
	1 - \frac{2\theta-1}{2\theta^{2}}
	\bigg\} \in {[0,1)}
\end{align}
and employing \eqref{lem:error:2} indicate
\begin{align}
	\notag
	&~\big(1 + 2(\lambda-L_{f})\theta\Delta \big)
	\E\big[|e_{-k\tau + t_{j+1}}^{-k\tau}|^{2}\big]
	\\&~+\notag
	2(p^{*}-1)\theta\Delta
	\E\big[|g_{-k\tau + t_{j+1}}^{-k\tau}
	- \hat{g}_{-k\tau + t_{j+1}}^{-k\tau}|^{2}\big]
	\\&~+\notag
	\theta^{2}\Delta^{2}
	\E\big[|(-A) e_{-k\tau + t_{j+1}}^{-k\tau}
	+ f_{-k\tau + t_{j+1}}^{-k\tau}
	- \hat{f}_{-k\tau + t_{j+1}}^{-k\tau}|^{2}\big]
	\\\leq&~\notag
	\overline{C_{\Delta}}
	\Big(\big(1 + 2(\lambda-L_{f})\theta\Delta \big)
	\E\big[|e_{-k\tau+t_{j}}^{-k\tau}|^{2}\big]
	\\&~+\notag
	2(p^{*}-1)\theta\Delta
	\E\big[|g_{-k\tau + t_{j}}^{-k\tau}
	- \hat{g}_{-k\tau + t_{j}}^{-k\tau}|^{2}\big]
	\\&~+\notag
	\theta^{2}\Delta^{2}
	\E\big[|(-A) e_{-k\tau + t_{j}}^{-k\tau}
	+ f_{-k\tau + t_{j}}^{-k\tau}
	- \hat{f}_{-k\tau + t_{j}}^{-k\tau}|^{2}\big]\Big)
	+
	C\Delta^{2},
\end{align}
which implies
\begin{align*}%
	\notag
	&~\big(1 + 2(\lambda-L_{f})\theta\Delta \big)
	\E\big[|e_{-k\tau + t_{j+1}}^{-k\tau}|^{2}\big]
	+
	2(p^{*}-1)\theta\Delta
	\E\big[|g_{-k\tau + t_{j+1}}^{-k\tau}
	- \hat{g}_{-k\tau + t_{j+1}}^{-k\tau}|^{2}\big]
	\\&~+\notag
	\theta^{2}\Delta^{2}
	\E\big[|(-A) e_{-k\tau + t_{j+1}}^{-k\tau}
	+ f_{-k\tau + t_{j+1}}^{-k\tau}
	- \hat{f}_{-k\tau + t_{j+1}}^{-k\tau}|^{2}\big]
	- \frac{C\Delta^{2}}{1-\overline{C_{\Delta}}}
	\\\leq&~\notag
	\overline{C_{\Delta}}\bigg(
	\big(1 + 2(\lambda-L_{f})\theta\Delta \big)
	\E\big[|e_{-k\tau+t_{j}}^{-k\tau}|^{2}\big]
	+
	2(p^{*}-1)\theta\Delta
	\E\big[|g_{-k\tau + t_{j}}^{-k\tau}
	- \hat{g}_{-k\tau + t_{j}}^{-k\tau}|^{2}\big]
	\\&~+\notag
	\theta^{2}\Delta^{2}
	\E\big[|(-A) e_{-k\tau + t_{j}}^{-k\tau}
	+ f_{-k\tau + t_{j}}^{-k\tau}
	- \hat{f}_{-k\tau + t_{j}}^{-k\tau}|^{2}\big]
	-
	\frac{C\Delta^{2}}{1-\overline{C_{\Delta}}}\bigg).
\end{align*}
It follows from an iterative argument and $e_{-k\tau}^{-k\tau}= 0$,
$f_{-k\tau}^{-k\tau} = \hat{f}_{-k\tau}^{-k\tau}$,
$g_{-k\tau}^{-k\tau} = \hat{g}_{-k\tau}^{-k\tau}$ that
\begin{align*}
	\notag
	&~\big(1 + 2(\lambda-L_{f})\theta\Delta \big)
	\E\big[|e_{-k\tau + t_{j}}^{-k\tau}|^{2}\big]
	+
	2(p^{*}-1)\theta\Delta
	\E\big[|g_{-k\tau + t_{j}}^{-k\tau}
	- \hat{g}_{-k\tau + t_{j}}^{-k\tau}|^{2}\big]
	\\&~+\notag
	\theta^{2}\Delta^{2}
	\E\big[|(-A) e_{-k\tau + t_{j}}^{-k\tau}
	+ f_{-k\tau + t_{j}}^{-k\tau}
	- \hat{f}_{-k\tau + t_{j}}^{-k\tau}|^{2}\big]
	- \frac{C\Delta^{2}}{1-\overline{C_{\Delta}}}
	\\\leq&~\notag
	\overline{C_{\Delta}}^{j}\bigg(
	\big(1 + 2(\lambda-L_{f})\theta\Delta \big)
	\E\big[|e_{-k\tau}^{-k\tau}|^{2}\big]
	+
	2(p^{*}-1)\theta\Delta
	\E\big[|g_{-k\tau}^{-k\tau}
	- \hat{g}_{-k\tau}^{-k\tau}|^{2}\big]
	\\&~+\notag
	\theta^{2}\Delta^{2}
	\E\big[|(-A) e_{-k\tau}^{-k\tau}
	+ f_{-k\tau}^{-k\tau}
	- \hat{f}_{-k\tau}^{-k\tau}|^{2}\big]
	-
	\frac{C\Delta^{2}}{1-\overline{C_{\Delta}}}\bigg)
	\\\leq&~\notag
	0.
\end{align*}
Together with
\begin{align*}
	\frac{1}{1-\overline{C_{\Delta}}}
	=
	\max\bigg\{
	\frac{1 + 2(\lambda-L_{f})\theta\Delta}
	{(\lambda-L_{f})\Delta},~ 2\theta,~
	\frac{2\theta^{2}}{2\theta-1}
	\bigg\}
	\leq
	\Delta^{-1}\max\bigg\{
	\frac{1 + 2(\lambda-L_{f})\theta}
	{\lambda-L_{f}},~ 2\theta,~
	\frac{2\theta^{2}}{2\theta-1}
	\bigg\},
\end{align*}
we derive
the desired result \eqref{eq:daoshuresult}. Concerning \eqref{eq:daoshuresultadditive}, it suffices to put \eqref{eq:RjrJFadditive} into \eqref{lem:error:lr} and continue the remainder steps as that to derive \eqref{eq:daoshuresult}. Thus we complete the proof.
\end{proof}

At this stage, it is necessary to compare our convergence result in Lemma \ref{thm:error} with that in \cite{wang2020meansquare}, where the authors establish the mean square convergence rates of ST methods with $\theta \in [1/2,1]$ on finite time interval for SDEs under a coupled monotonicity condition. In contrast, the convergence result for ST methods with critical parameter $\theta = \frac{1}{2}$ does not include in Lemma \ref{thm:error}, which can be attributed to the infiniteness of time interval and the superlinear growth of coefficients (see \cite{rong2020numerical} for SDEs with global Lipschitz coefficients). Actually, the upper bound on the convergence error of ST method with $\theta = \frac{1}{2}$ grows exponentially with the length of time interval, making a convergence analysis on infinite time interval impossible.


Armed with the above preparations, we are able to reveal the main result of this work.

\begin{theorem}\label{cor:error}
	Let $X_{t}^{*}$ be the random periodic solution of SDE \eqref{eq:SDE} and $\hat{X}_{t}^{*}$ the random periodic solution of the stochastic theta approximation \eqref{eq:BEM}. Suppose that Assumption \ref{asm:L} holds. Then there exists $C > 0$, independent of $\Delta$, such that
	\begin{equation}\label{eq:convergenceorder}
		|X_{t}^{*}-\hat{X}_{t}^{*}|_{L^{2}(\Omega;\R^{d})}
		\leq
		C\Delta^{\frac{1}{2}}.
	\end{equation}
	If in addition Assumption \ref{asm:Ladditive} holds, then there exists $C > 0$, independent of $\Delta$, such that
	\begin{equation}\label{eq:convergenceorderadditive}
		|X_{t}^{*}-\hat{X}_{t}^{*}|_{L^{2}(\Omega;\R^{d})}
		\leq
		C\Delta.
	\end{equation}
\end{theorem}

\begin{proof}
Based on \eqref{thm:ep:rs} and \eqref{thm:ep:rsrere}, we utilize Lemma \ref{thm:error} and
\begin{equation*}
	|X_{t}^{*}-\hat{X}_{t}^{*}|_{L^{2}(\Omega;\R^{d})}
	\leq
	\varlimsup_{k \to \infty}
	\big(
	|X_{t}^{*}-X_{t}^{-k\tau}|_{L^{2}(\Omega;\R^{d})}
	+
	|X_{t}^{-k\tau}-\hat{X}_{t}^{-k\tau}|_{L^{2}(\Omega;\R^{d})}
	+
	|\hat{X}_{t}^{-k\tau}-\hat{X}_{t}^{*}|_{L^{2}(\Omega;\R^{d})}
	\big)
\end{equation*}
to obtain \eqref{eq:convergenceorder} and \eqref{eq:convergenceorderadditive}
and finish the proof.
\end{proof}

\section{Numerical experiments}
\label{sec:numerical}
Some numerical experiments will be performed to illustrate the previous theoretical results in this section. Let us focus on the following one-dimensional SDE
\begin{align}\label{expl:one}
	\text{d}X^{t_0}_{t}
	=
	\big(-\lambda X^{t_0}_{t}
	- a(X^{t_0}_{t})^3 (1+\sin(\pi t))\big)\text{d}t
	+
	\big( b + cX^{t_0}_{t}
	+ d (X^{t_0}_{t})^{2}(1+\sin(\pi t)) \big)\text{d}W_{t}
\end{align}
for all $t \geq t_{0}$, where $\lambda, a > 0$ and $b,c,d \in \R$. Compared with \eqref{eq:SDE}, we know that
\begin{align}\label{expl:fg}
	A = \lambda,
	\quad
	f(t,x)
	=
	-ax^3 (1+\sin(\pi t)),
	\quad
	g(t,x)
	=
	b + cx + dx^2 (1+\sin(\pi t)),
	\quad
	x \in \R, t \in \R.
\end{align}
Obviously, \rm{(A1)} holds. Besides, $f$ and $g$ are continuous and periodic in time with period $\tau = 2 > 0$, i.e., \rm{(A2)} holds. Noting that for any $x,y \in \R$ and $s,t \in \R$, we have
\begin{align}\label{eq:ftxfsy}
	|f(t,x)-f(s,y)|
	=&~\notag
	|-a(x^3-y^3)(1+\sin(\pi t))
	- ay^3(\sin(\pi t)-\sin(\pi s))|
	\\\leq&~
	2a(|x|^2 + |x||y|+ |y|^2)|x-y|
	+
	a\pi|y|^{3}|t-s|
	\\\leq&~\notag
	2a(1 + |x| + |y|)^2|x-y|
	+
	a\pi(1 + |x| + |y|)^3|t-s|,
\end{align}
and
\begin{align*}
	&~(x-y)(f(t,x)-f(s,y))
	\\=&~
	(x-y)\big(-a(x^3-y^3)(1+\sin(\pi t))
	- ay^3(\sin(\pi t)-\sin(\pi s))\big)
	\\\leq&~
	-a|x-y|^{2}(x^2+xy+y^2)(1+\sin(\pi t))
	+ a\pi(|x||y|^3+|y|^{4})|t-s|,
\end{align*}
as well as
\begin{align*}
	|g(t,x)-g(s,y)|^{2}
	=&~
	|c(x-y) + d(x^2-y^2)(1+\sin(\pi t))
	+ dy^2(\sin(\pi t) - \sin(\pi s))|^{2}
	\\\leq&~
	3c^{2}|x-y|^{2}
	+
	6d^{2}|x-y|^{2}(x+y)^2(1+\sin(\pi t))
	+
	6{\pi} d^{2}|y|^4|t-s|.
\end{align*}
By requiring
\begin{equation}\label{eq:12d2p*1}
	12d^{2}(p^*-1) \leq a,
\end{equation}
one can show that for any $x,y \in \R$ and $s,t \in \R$,
\begin{align*}
	&~(x-y)(f(t,x)-f(s,y))+(p^*-1)|g(t,x)-g(s,y)|^{2}
	\\\leq&~
	\big(3c^{2}(p^*-1)
	-
	\big(a(x^2+xy+y^2)
	-
	6d^{2}(p^*-1)(x+y)^2\big)(1+\sin(\pi t))\big)|x-y|^{2}
	\\&~+
	a\pi(|x||y|^3+|y|^{4})|t-s|
	+
	6{\pi}d^{2}(p^*-1)|y|^4|t-s|
	\\\leq&~
	3c^{2}(p^*-1)|x-y|^{2}
	+
	\big(2a\pi + 6{\pi}d^{2}(p^*-1)\big)
	(1 + |x|^{4} + |y|^{4})|t-s|,
\end{align*}
which together with \eqref{eq:ftxfsy} indicates that \rm{(A3)} holds with $L_{f} = 3c^{2}(p^*-1)$ and $\gamma = 3$. In combination with \eqref{eq:12d2p*1}, $L_{f} \in (0,\lambda)$ and {$p^{*} > 5\gamma$,} we choose
\begin{align*}
	p^{*} = 21,     \quad \lambda = 5\pi,
	\quad a = 3,    \quad b = 1.5,
	\quad c = 0.5,  \quad d = 0.1
\end{align*}
such that all conditions in Assumption \ref{asm:L} are fulfilled. It follows that \eqref{expl:one} admits a unique random periodic solution. 
Before performing experiments, we emphasize that in what follows the Newton iterations with precision $10^{-5}$ are employed to solve the nonlinear equation arising from the implementation of the implicit ST method in every time step.

\begin{figure}[!htbp]
	\begin{center}
		\subfigure[$\theta = 0.75$]
		{\includegraphics[width=0.45\textwidth]{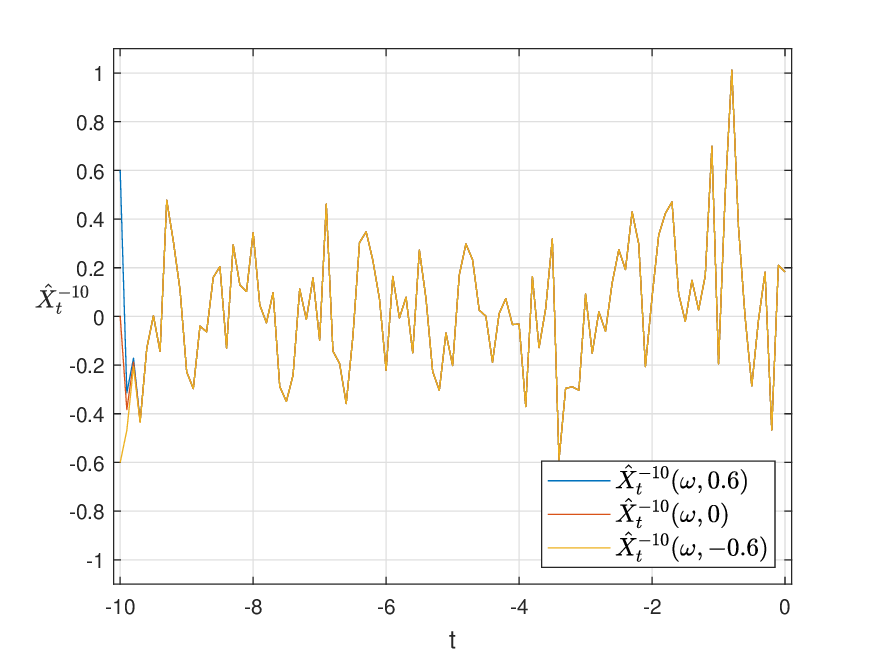}}
		\subfigure[$\theta = 1$]
		{\includegraphics[width=0.45\textwidth]{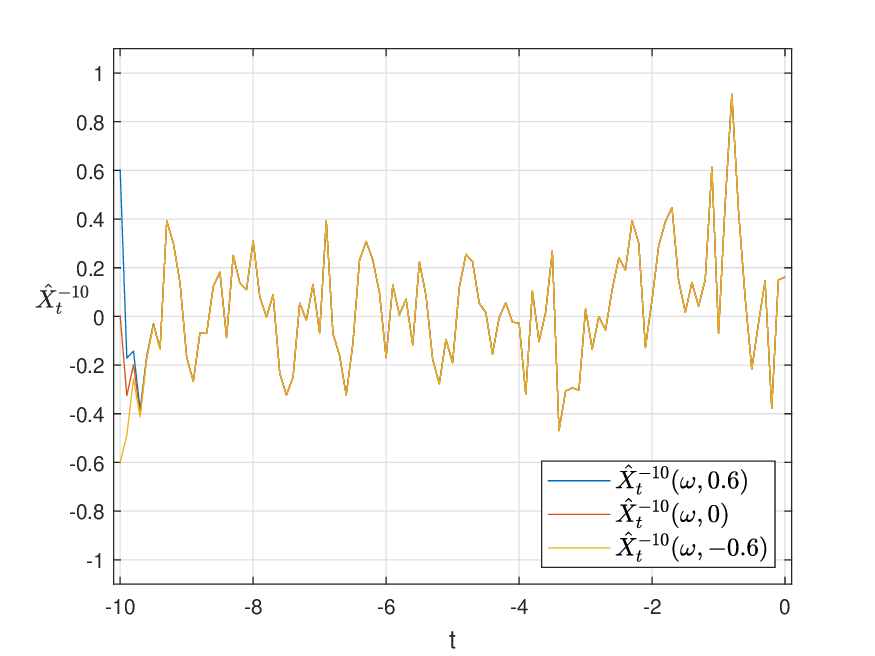}}
		\centering\caption{Random periodic solution does not dependent on the initial values}
		\label{fig:initial}
	\end{center}
\end{figure}

\begin{figure}[!htbp]
	\begin{center}
		\subfigure[$\theta = 0.75$]
		{\includegraphics[width=0.45\textwidth]{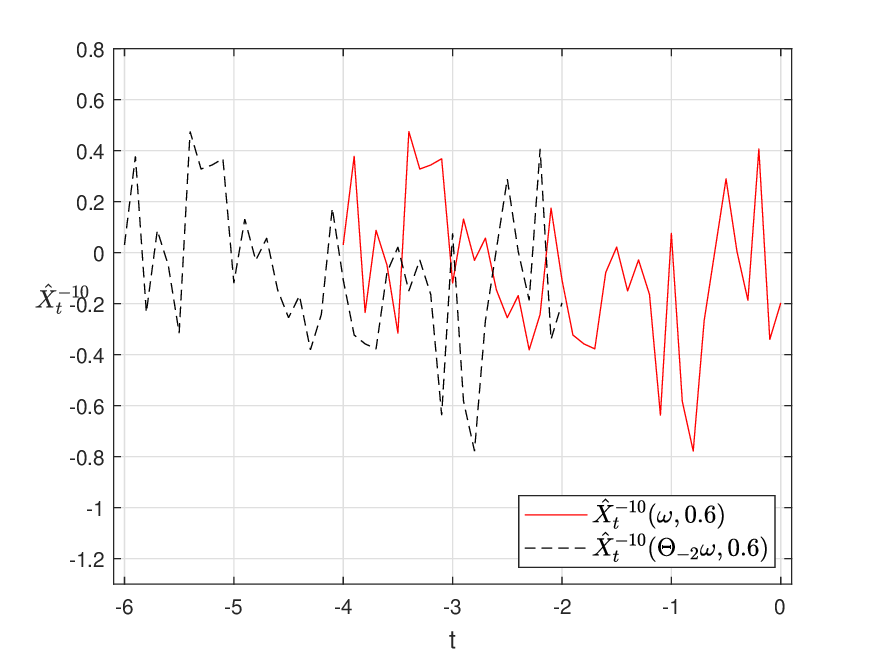}}
		\subfigure[$\theta = 1$]
		{\includegraphics[width=0.45\textwidth]{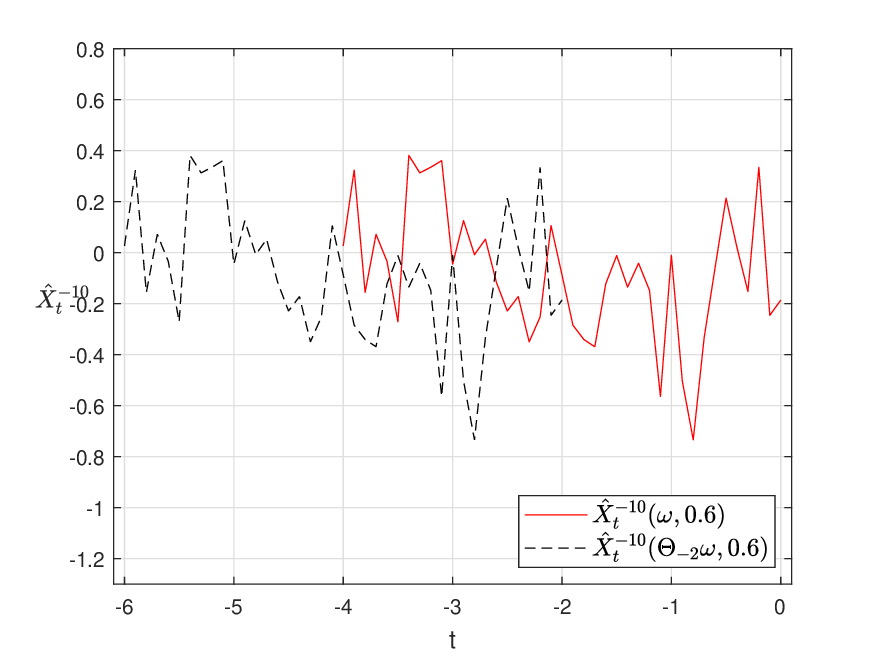}}
		\centering\caption{Validate periodicity based on $\hat{X}_{t}^{*}(\Theta_{-\tau}\omega)
			= \hat{X}_{t-\tau}^{*}(\omega)$}
		\label{fig:pI}
	\end{center}
\end{figure}

\begin{figure}[!htbp]
	\begin{center}
		\subfigure[$\theta = 0.75$]
		{\includegraphics[width=0.45\textwidth]{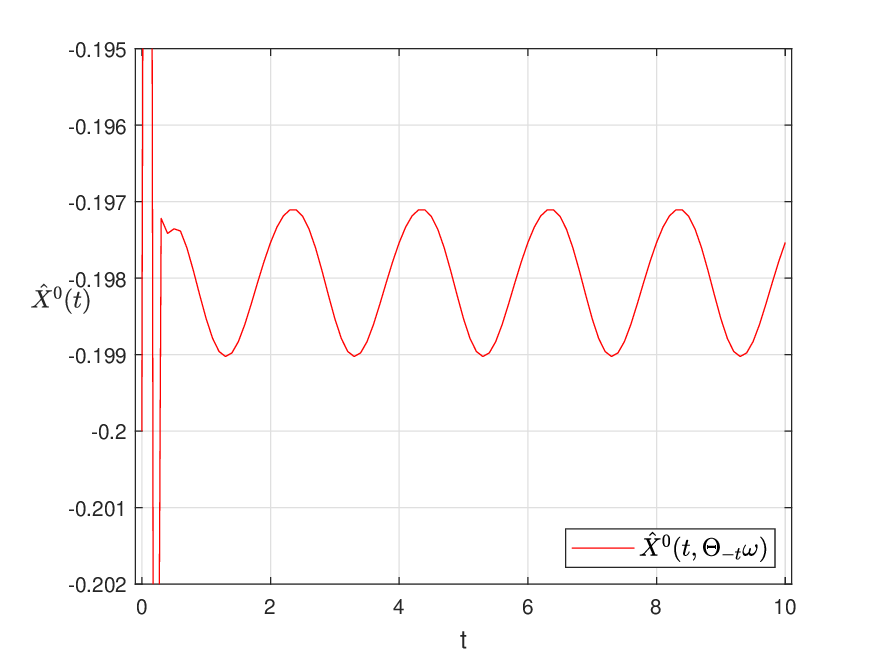}}
		\subfigure[$\theta = 1$]
		{\includegraphics[width=0.45\textwidth]{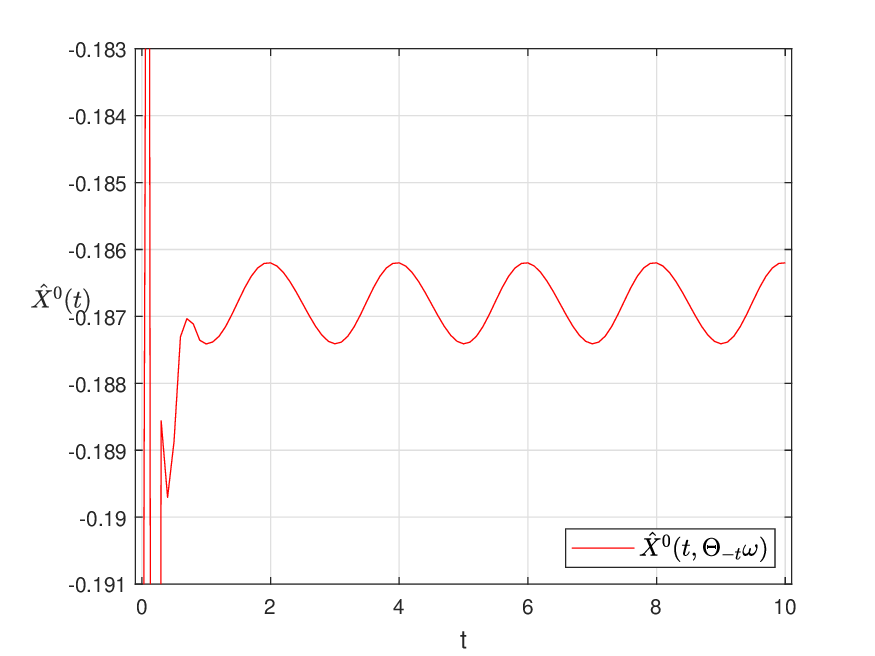}}
		\centering\caption{Validate periodicity via the periodicity of $\hat{X}^{0}(t,\Theta_{-t}\omega)$}
		\label{fig:pII}
	\end{center}
\end{figure}

Theorem \ref{thm:nump} indicates that for each $\theta \in (1/2,1]$, the corresponding ST method applied to \eqref{expl:one} admits a unique random periodic solution. Let us first numerically verify that the random periodic solution of ST method does not dependent on the initial values. To this end, we choose three different initial values $\xi = \pm0.6,0$ for $\theta = 0.75,1$, the corresponding paths $\hat{X}_{t}^{-10}(\omega,0.6)$,
$\hat{X}_{t}^{-10}(\omega,0)$ and $\hat{X}_{t}^{-10}(\omega,-0.6)$ are simulated on the time interval $[-10,0]$ with stepsize $\Delta = 0.1$; see
Figure \ref{fig:initial}. It can be seen that for each $\theta$, the three paths coincide after a very short time, which means that taking pull-back time $-10$ is sufficient to obtain a good convergence to the random periodic paths for $t \geq -8$.
Following the arguments in \cite{feng2017numerical}, we will verify the periodicity in two ways. Based on $\hat{X}_{t}^{*}(\Theta_{-\tau}\omega) = \hat{X}_{t-\tau}^{*}(\omega)$, we will plot the paths $\hat{X}_{t}^{-10}(\omega,0.6), t \in [-4,0]$ and $\hat{X}_{t}^{-10}(\Theta_{-2}\omega,0.6), t \in [-6,-2]$ for $\theta = 0.75,1$ with stepsize $\Delta = 0.1$ in Figure \ref{fig:pI}, which shows that the two segmented processes resemble each other with the period $\tau = 2$. At last, we check whether or not $\hat{X}_{t}^{*}(t,\Theta_{-\tau}\omega)$ is periodic with period $\tau = 2$. For this purpose, the path $X^{0}(t,\Theta_{-t}\omega), t \in [0,10]$ with $X^{0}(0,\Theta_{-0}\omega) = -0.2$ is given in Figure \ref {fig:pII} with stepsize $\Delta = 0.1$. Then we obtain a periodic pull-back path with period $\tau = 2$, which validates the random periodicity of the original path.

\begin{figure}[!htbp]
	\begin{center}
		\subfigure[$\theta = 0.75$]
		{\includegraphics[width=0.45\textwidth]{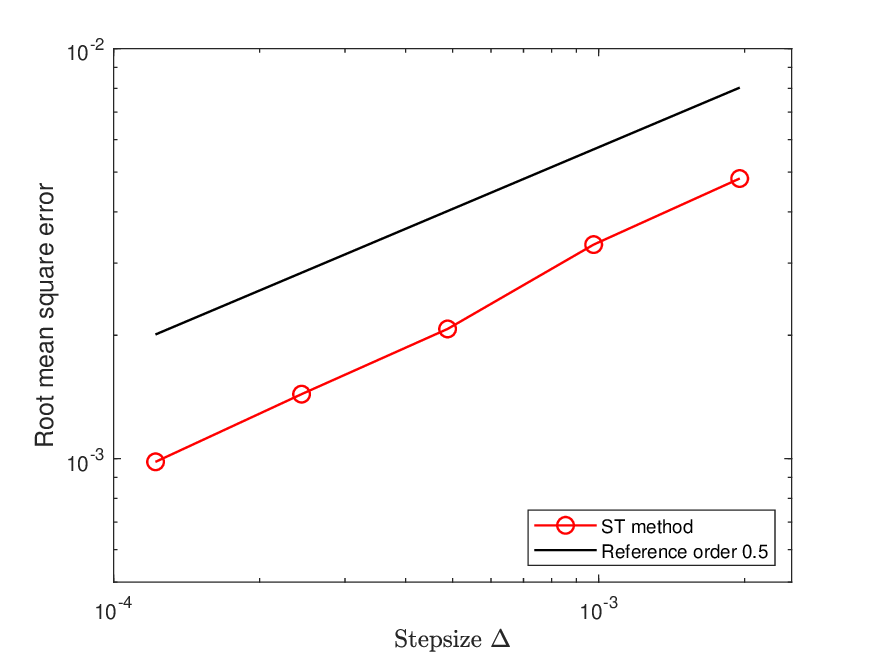}}
		\subfigure[$\theta = 1$]
		{\includegraphics[width=0.45\textwidth]{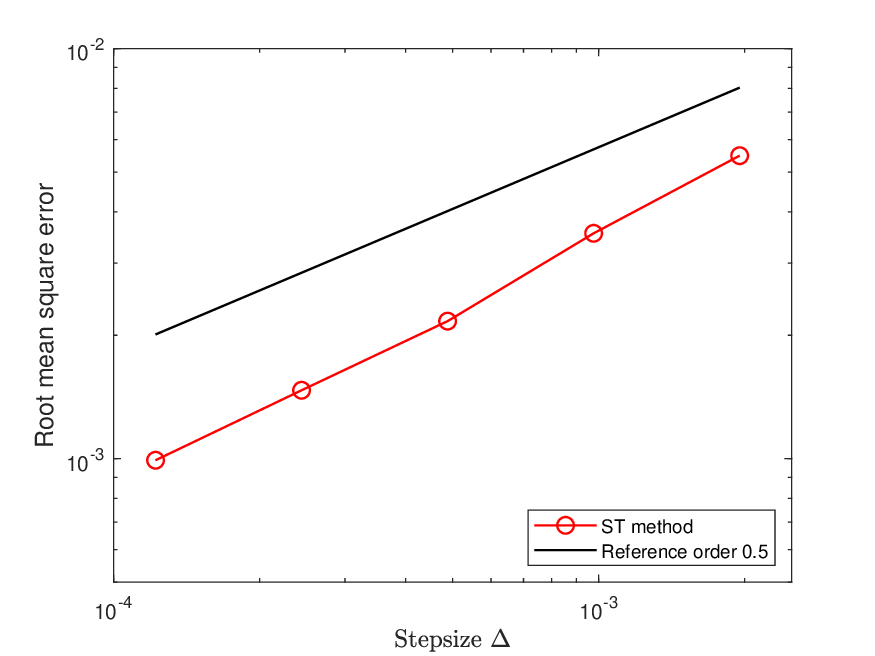}}
		\centering\caption{Mean square convergence orders for \eqref{expl:one}}
		\label{fig:pconverge}
	\end{center}
\end{figure}

Theorem \ref{cor:error} shows that for every $\theta \in (1/2,1]$, the random periodic solution converges to that of \eqref{expl:one} with order $1/2$ in the mean square sense. To support this result numerically, we identify the unavailable exact solution with a numerical approximation generated by the ST method with a fine stepsize $\Delta = 2^{-16}$ on the time interval $[-10,10]$. The other numerical approximations are calculated with five different equidistant stepsize $\Delta = 2^{-i}, i = 10,\cdots,14$. Here the expectations are approximated by the Monte Carlo simulation with 500 different Brownian paths. Figure \ref{fig:pconverge} shows that for different $\theta$, the root mean square error line and the reference line appear to parallel to each other, indicating that the mean square convergence rate of the ST method is $1/2$. A least square fit indicates that the slope of the line for the ST method is $0.5804$ for $\theta = 0.75$ and $0.6208$ for $\theta = 1$, identifying the theoretical result in Theorem \ref{cor:error}.

\begin{figure}[!htbp]
	\begin{center}
		\subfigure[$\theta = 0.75$]
		{\includegraphics[width=0.45\textwidth]{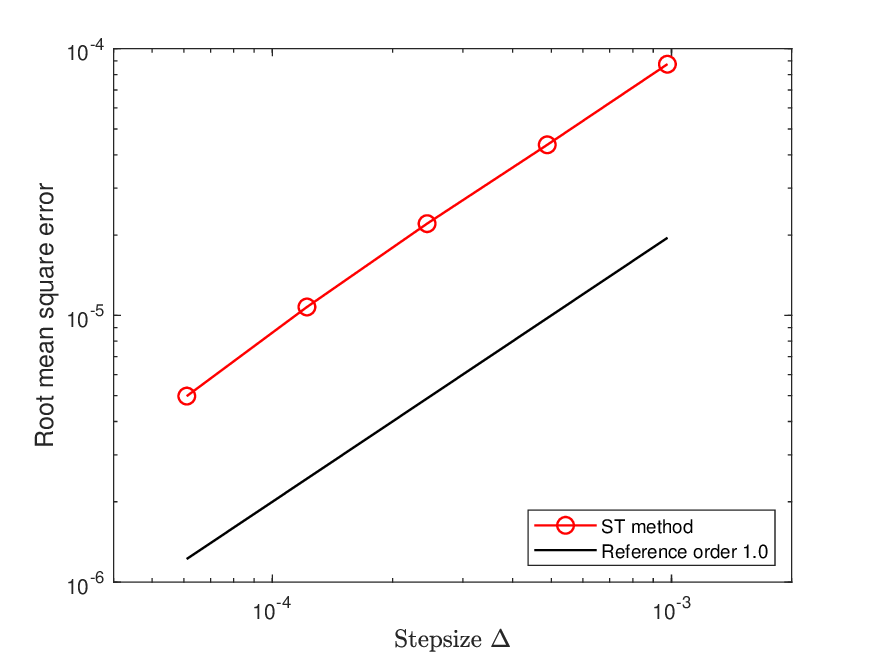}}
		\subfigure[$\theta = 1$]
		{\includegraphics[width=0.45\textwidth]{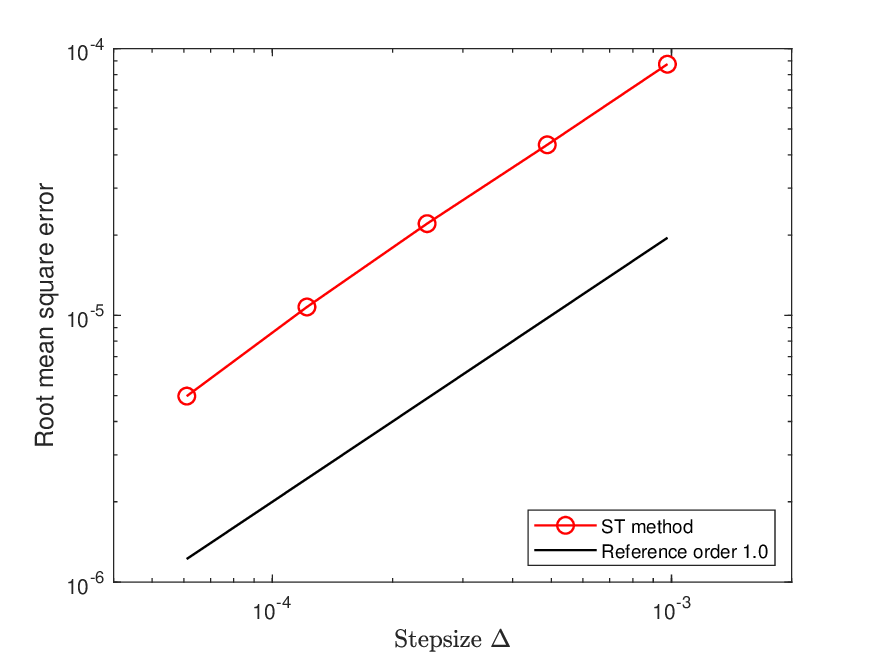}}
		\centering\caption{Mean square convergence orders for \eqref{eq:exadditive}}
		\label{fig:pconvergeadditive}
	\end{center}
\end{figure}

For the additive case, we consider the following one-dimensional SDE
\begin{equation}\label{eq:exadditive}
	\diff{X_{t}^{t_{0}}}
	=
	-10\pi X_{t}^{t_{0}}\diff{t}
	+
	\sin(2 \pi t)\diff{t} + 0.05\diff{W_{t}},
    \quad t \geq t_{0}.
\end{equation}
One can check that the associated period is $1$ and Assumptions \ref{asm:L} and \ref{asm:Ladditive} are fulfilled. It is worth noting that \eqref{eq:exadditive} has been considered in \cite{wu2023backward} to test the periodicity and mean square convergence of the backward Euler method. Thus for simplicity, we only verify the mean square order of ST methods on the time interval $[-10,10]$ with $\theta = 0.75,1$ under the above settings for convergence test of \eqref{expl:one}. Figure \ref{fig:pconvergeadditive} shows that the slopes of the error lines and the reference lines match well, indicating that the ST methods have strong rates of {order one} in additive case. Actually, least square fits produce rate $1.0289$ for $\theta = 0.75$ and rate $1.0289$ for $\theta = 1$. These facts coincide with the previous theoretical result in Theorem \ref{cor:error}.


%
%
%
%
%
%
%
%
%
%
%
%
%
%
%
%


\end{document}